\documentclass[12pt]{amsart}
\usepackage{amsmath}
\usepackage{graphicx}
\usepackage{hyperref}
\usepackage[latin1]{inputenc}
\usepackage{mathtools}
\usepackage{amsfonts}
\usepackage{amssymb}
\usepackage[T1]{fontenc}
\usepackage{amsthm}
\usepackage{fullpage}

\newtheorem{theorem}{Theorem}[section]

\newtheorem{lemma}[theorem]{Lemma}
\newtheorem{Conjecture}[theorem]{Conjecture}

\theoremstyle{definition}
\newtheorem{definition}{Definition}[section]
 
\theoremstyle{remark}

\numberwithin{equation}{section}

\newcommand{\F}{\mathbb{F}_q}
\newcommand{\Fm}{\mathbb{F}_{q^m}}
\newcommand{\M}{\mathfrak{M}}
\newcommand{\s}{\mathfrak{S}}
\newcommand{\vt}{\vartheta}
\newcommand{\si}{\sigma}

\title{On the existence of primitive normal  elements of rational form over finite fields of even characteristic}


\author{Himangshu Hazarika}
\address{Department of Mathematical Sciences, Tezpur University, Assam, India}
\email{diku\_95@tezu.ernet.in}

\author{Dhiren Kumar Basnet}
\address{Department of Mathematical Sciences, Tezpur University, Assam, India}
\email{dbasnet@tezu.ernet.in}

\author{Giorgos Kapetanakis}
\address{Department of Mathematics and Applied Mathematics, University of
Crete, Voutes Campus, 70013 Heraklion, Greece}
\email{gnkapet@gmail.com}

\thanks{This work was funded by the Council of Scientific and Industrial Research, New Delhi, Government of India's research grant no.~09/796(0099)/2019-EMR-I}

\begin{document}
\begin{abstract}
 Let $q$ be an even prime power and $m\geq2$ an integer. By $\mathbb{F}_q$, we denote the finite field of order $q$ and by $\mathbb{F}_{q^m}$ its extension degree $m$. In this paper we investigate the  existence of a primitive normal pair $(\alpha, \, f(\alpha))$, with $f(x)= \frac{ax^2+bx+c}{dx+e} \in \mathbb{F}_{q^m}(x)$ where the rank of the matrix 
 $F= \begin{pmatrix}a \, &b\, & c\\
0\, &d \, &e
\end{pmatrix}$ $\in M_{2 \times 3}(\Fm)  $ is 2. Namely, we establish sufficient conditions to show that nearly all fields of even characteristic possess such elements, except for $\begin{pmatrix} 1 \, &1 \, & 0\\
0\, &1 \, &0
\end{pmatrix}$ if $q=2$ and $m$ is odd, and then we provide an explicit small list of possible and genuine exceptional pairs $(q,m)$.
 \end{abstract}
 
\maketitle

\section{Introduction}
Given an even prime power  $q$  and an integer $m\geq2$, we  denote by $\mathbb{F}_q$, the finite field of order $q$ and by $\mathbb{F}_{q^m}$ its extension field of degree $m$. A generator of the (cyclic)  multiplicative group $\mathbb{F}^*_{q^m}$ is called {\em primitive}. It is well-known that, for any finite field $\F$, there are $\phi(q-1)$ primitive elements, where $\phi$ is Euler's phi-function. Further, an $\mathbb{F}_q$-basis of $\mathbb{F}_{q^m}$ of the form $\{\alpha,\alpha^q, \alpha^{q^2},\ldots, \alpha^{q^{m-1}}\}$ is called a {\em normal basis} and $\alpha$ is called {\em normal} or {\em free}.

 The readers are referred to \cite{10} and the references therein for the existence of both primitive and free elements. The simultaneous occurrence of primitive and free elements in $\mathbb{F}_{q^m}$ is given by the following theorems.
\begin{theorem}[Primitive normal basis theorem, \cite{3}]
For any prime power $q$ and positive integer $m$, the finite field $\mathbb{F}_{q^m}$ contains some element which is simultaneously primitive and free.
\end{theorem}

At first, this result was proved by Lenstra and Schoof in \cite{2}. Later on, by implementing a sieving technique that was initially introduced by Cohen \cite{16}, Cohen and Huczynska \cite{3} provided a computer-free proof.   
   
\begin{theorem}[Strong primitive normal basis theorem \cite{4}]
In $\mathbb{F}_{q^m}$, there exists some element $\alpha$ such that both $\alpha$ and $\alpha^{-1}$ are primitive and free, unless $(q,m)$ is $ (2,3),(2,4),(3,4),(4,3)$ or $(5,4)$. 
\end{theorem}
Tian and Qi were the first to provide this result in \cite{5}, for $m\geq 32$. Later on Cohen and Huczynska \cite{4} completed the proof up to the above form, again by using their sieving technique. 
 
The next theorems, which extend to rational functions, were given by Kapetanakis \cite{27,8} by employing the aforementioned sieving technique.
  
\begin{theorem}[\cite{27}]
For odd prime power $q\geq 23$, an integer $m\geq 17$ and $A = \begin{pmatrix}a \, &b\\
c \, &d
\end{pmatrix}$ $\in GL_2(\F) $, with the condition that if $A$ has exactly two non-zero entries and $q$ is odd, then the quotient of these entries is a square in $\Fm$. Then there exists some $\alpha\in\Fm$ such that both $\alpha$ and $\frac{a\alpha+b}{c\alpha+d}$ are simultaneously primitive and free.\label{T1}
\end{theorem}

\begin{theorem}[\cite{8}]
Let $q$ be a prime power, $n\geq 2$ an integer and some matrix $M = \begin{pmatrix}a \, &b\\
c \, &d
\end{pmatrix}$ $\in GL_2(\F) $, where $M \neq \begin{pmatrix}1 \, &1\\
0 \, &1
\end{pmatrix}$ if $q=2$ and $m$ is odd. There exists some primitive $\alpha\in \Fm$ such that $\alpha$ and $(a\alpha+b)/(c\alpha+d)$ are both simultaneously normal elements of $\Fm$ over $\F$.

\end{theorem}

The existence of a primitive element $\alpha \in \F$ such that $f(\alpha)$ is also primitive for an arbitrary quadratic in $\F[x]$ has been completely resolved in \cite{1}.
\begin{theorem}[\cite{1}]
For all $q>211$, there always exists an element $\alpha\in \mathbb{F}_{q^m}$ such that $\alpha$ and $f(\alpha)$ are both primitive, where $f(x)=ax^2+bx+c$ with $b^2-4ac\neq0$.
\end{theorem}

In this paper, we extend of Theorem \ref{T1}. We solve the existence question for elements $\alpha$ of $\Fm$  that both $\alpha$ and $f(\alpha)$  are simultaneously primitive and normal over $\F$, where $f(x)=\frac{ax^2+bx+c}{dx+e} \in \Fm(x)$ such that the $F= \begin{pmatrix}a \, &b\, & c\\
0\, &d \, &e
\end{pmatrix}$ $\in M_{2 \times 3}(\Fm)  $ has rank 2. For $a=0$, the results are already discussed in \cite{10}, hence throughout this paper we assume $a \neq 0$. We call the pair $(q,m)$ a \emph{primitive normal} pair if the field $\Fm$ contains such elements. In particular, we prove the following results, where $m^\prime$ is odd such that $m=2^k m^\prime$, $k\geq 0$. 
\begin{theorem} \label{main1}
For the finite field $\Fm$ of even characteristic, suppose $m^\prime$ is such that $m^\prime|q-1$. Then there exists an element $\alpha$ in $\mathbb{F}_{q^m}$, such that both $\alpha$ and  $f(\alpha)$ are simultaneously primitive normal in $\mathbb{F}_{q^m}$ over $\F$, where $f(x)=\frac{ax^2+bx+c}{dx+e}$, with $a,b,c,d,e\in \mathbb{F}_{q^m}$, $a\neq 0$, and $dx+e\neq 0$ unless $(q,m)$ is one of the pairs $(2,2)$, $(2,4)$, $(2,8)$, $(2,16)$, $(4,2)$, $(4,3)$,  $(4,4)$, $(4,6)$, $(4,8)$, $(4,12)$, $(8,2)$, $(8,4)$, $(8,7)$, $(8,8)$, $(8,14)$, $(16,2)$, $(16,3)$, $(16,4)$, 
$(16,5)$, $(16,6)$, $(16,15)$, $(32,2)$, $(64,2)$,  $(64,4)$, $(128,2)$, $(256,2)$, $(512,2)$ or $(1024,2)$. 
\end{theorem}
 
\begin{theorem} \label{main2}
Let $\mathbb{F}_{q^m}$ be a finite field of even characteristic and $m'\nmid q-1$. Then there exists an element $\alpha$ in $\mathbb{F}_{q^m}$, such that both $\alpha$ and  $f(\alpha)$ are simultaneously primitive normal in $\mathbb{F}_{q^m}$ over $\F$, where $f(x)=\frac{ax^2+bx+c}{dx+e}$, with $a,b,c,d,e\in \mathbb{F}_{q^m}$, $a\neq 0$, and $dx+e\neq 0$ unless $(q,m)$ is one of the pairs
$(2,3)$, $(2,5)$, $(2,6)$, $(2,7)$, $(2,9)$, $(2,10)$, $(2,11)$, $(2,12)$, $(2,13)$, $(2,14)$, $(2,15)$, $(2,18)$, $(2,20)$, $(2,21)$, $(2,24)$, $(2,30)$, $(4,5)$, $(4,7)$, $(4,9)$, $(4,10)$, $(8,3)$, $(8,5)$, $(8,6)$ or $(32,3)$.
\end{theorem}

In addition, we employ explicit computational methods and show that the pairs $(2, 2)$, $(2, 3)$, $(2, 4)$, $(2, 5)$, $(2, 6)$, $(4, 2)$ and $(4, 6)$ appearing above are genuine exceptions, while, based on computational evidence, we conjecture, that they are actually the only genuine exceptions, see Conjecture~\ref{conj1}.

This work is heavily influenced by the the work of Lenstra and Schoof~\cite{2}, while character sums plays a very crucial role.
Further, we adjust the sieving technique, provided by Cohen and Huczynska \cite{3, 4}, in our setting.

In Section~\ref{sec3}, we estimate a lower bound for existence of primitive normal pair. Then in Section~\ref{sec4},  by using "the prime sieve technique", we weaken the sufficient condition for more efficient results.  
In Section~\ref{sec5}, we apply the existence conditions on fields of even characteristic for each and every possible case and complete the proofs of Theorems~\ref{main1} and \ref{main2}. In Section~\ref{sec:computations}, we employ computers to further investigate the actual situation with the pairs posing as possible exceptions on Theorems~\ref{main1} and \ref{main2}.
We conclude this work with the statement of two related conjectures in Section~\ref{sec6}.

\section{Preliminaries} 
Under the rule $f\circ \alpha=\overset{n}{\underset{i=1}{\sum}} a_i\alpha^{q^i}$ and 
 $f\,= \overset{n}{\underset{i=1}{\sum}}a_ix^i\thinspace \in \mathbb{F}_{q}[x]$ for $\alpha\in \mathbb{F}_{q^m}$; the additive group of $\mathbb{F}_{q^m}$ is an $\mathbb{F}_{q}[x]$-module.  The $\mathbb{F}_q$-order of $\alpha \in \mathbb{F}_{q^m}$, is the monic $\mathbb{F}_q$-divisor $g$ of $x^m-1$ of minimal degree such that $g\circ \alpha=0$, 
  which we define as \emph{Order of $\alpha$} and denote by $\mathrm{Ord}(\alpha)$. It is clear that the free elements of $\mathbb{F}_{q^m}$ are exactly those of Order $x^m-1$. 
 
  The multiplicative order for $\alpha\in \mathbb{F}^*_{q^m}$ is denoted by $\mathrm{ord}(\alpha)$ and $\alpha$ is primitive if and only if $\mathrm{ord}(\alpha)=q^m-1$. Furthermore, it follows from the definitions that $q^m-1$ and $x^m-1$ can be freely replaced by their radicals $q_0$ and $f_0\,:=\, x^{m_0}-1 $ respectively, where $m_0$ is such that $m=m_0p^a$, where $a$ is a non negative integer and $\gcd(m_0,p)=1$.
  
   Throughout this section we present a couple of functions that characterize primitive and free elements. To represent those functions, the idea of character of finite abelain group is necessary.
 \begin{definition} 
 Let $G$ be a finite abelian group. A character $\chi$ of $G$ is a group homomorphism from $G$ into the group $S^1:= \{z \in \mathbb{C} : |z| = 1\}$. The characters of $G$ form a group under multiplication called the \emph{dual group} or \emph{character group} of $G$, that is denoted by $\widehat{G}$ and is isomorphic to $G$. The character $\chi_0$ defined as $\chi_0(a)=1$ for all $a \in G$ is called the \emph{trivial character} of $G$.
  \end{definition}
 
  In a finite field $\mathbb{F}_{q^m}$, the additive group $\mathbb{F}_{q^m}$ and the multiplicative group $\mathbb{F}^*_{q^m}$ are abelian groups. Throughout this paper we call the characters of the additive group $\mathbb{F}_{q^m}$ \emph{additive characters} and the characters of $\mathbb{F}^*_{q^m}$ \emph{multiplicative characters}. Multiplicative characters are extended from $\mathbb{F}^*_{q^m}$ to $\mathbb{F}_{q^m}$ by the  rule
                          $\chi(0)=\begin{cases}
                         0,& \text{if } \chi\neq\chi_0,\\
                         1,& \text{if } \chi=\chi_0. 
                         \end{cases} $ 
                      Furter, since $\widehat{\mathbb{F}^*_{q^m}} \cong \mathbb{F}^*_{q^m}$, $\widehat{\mathbb{F}^*_{q^m}}$ is cyclic and for any divisor $d$ of $q^m-1$ there are exactly $\phi(d)$ characters of order $d$ in $\widehat{\mathbb{F}^*_{q^m}}$.

   Let $e|q^m-1$, then $\alpha \in \Fm$ is called \emph{$e$-free} if $d|e$ and $\alpha= \beta^d$, for some $\beta\in \Fm$ implies $d=1$. Furthermore $\alpha$ is primitive if and only if $\alpha= \beta^d$, for some $\beta \in \Fm$ and $d|q^m-1$ implies $d =1$.

 For any $e|q^m-1$, following Cohen and Huczynska \cite{3, 4}, we express  the characteristic function for the subset of $e$-free elements of $\mathbb{F}^*_{q^m}$ as follows:
 $$\rho_e: \alpha\mapsto\theta(e)\underset{d|e}{\sum}(\frac{\mu(d)}{\phi(d)}\underset{\chi_d}{\sum}\chi_d(\alpha)),$$
     where $\theta(e):=\frac{\phi(e)}{e}$, $\mu$ is the M\"obius function and $\chi_d$ stands for any multiplicative character of order $d$.
   For any $e|q^m-1$, we use ``integral'' notation due to Cohen and Huczynska \cite{3, 4}, for weighted sums as follows
\begin{align*}
\underset{d|q^m-1}{\int}\chi_d:= \underset{d|q^m-1}{\sum}\frac{\mu(d)}{\phi(d)}\underset{\chi_d}{\sum}\chi_d.
\end{align*}

 Then the  characteristic function for the subset of $e$-free elements of $\mathbb{F}^*_{q^m}$ becomes,
\begin{align*}
\rho_e: \alpha\mapsto\theta(e)\underset{d|e}{\int}\,\chi_d(\alpha).
\end{align*}
    
     Again, for any monic $\mathbb{F}_q$-divisor $g$ of $x^m-1$, a typical additive character $\psi_g$ of \emph{$\mathbb{F}_q$-order} $g$ is one such that $\psi_g \circ g$ is the trivial character of $\mathbb{F}_{q^m}$ and $g$ is of minimal degree satisfying this property. 
    It is well-known that there are $\Phi(g)$ characters $\psi_g$, where $\Phi(g)= (\mathbb{F}_q[x]/g\mathbb{F}_q[x])^*$ is the analogue of the Euler function over $\mathbb{F}_{q}[x]$. 
    
    Then the characteristic function for the set of $g$-free elements in $\mathbb{F}_{q^m}$, for any $g|x^m-1$ is given by
    $$\kappa_g :\alpha\mapsto\Theta(g)\underset{f|g}{\sum}(\frac{\mu^\prime(f)}{\Phi(f)}\underset{\psi_f}{\sum}\psi_f(\alpha)),$$
    where $\Theta(g):= \frac{\Phi(g)}{q^{deg(g)}}$, the  sum runs over all additive characters $\psi_f$ of $\mathbb{F}_q$-order g and $\mu^\prime$ is the analogue of the M\"obius function which is defined as follows: 
      $$\mu^\prime(g)=\begin{cases}
                         (-1)^s , & \text{if $g$ is the product of $s$ distinct irreducible monic polynomials}, \\
                        0 , &\text{otherwise.}\\ 
                         \end{cases} $$
 
      We use the ``integral'' notation for weighted sum of additive characters as follows
\begin{align*}
\underset{f|g}{\int}\psi_f:=\underset{f|g}{\sum}\frac{\mu^\prime(f)}{\Phi(f)}\underset{\psi_f}{\sum}\psi_f.
\end{align*}     
   Then the characteristic function for the set of $g$-free elements in $\mathbb{F}_{q^m}$, for any $g|x^m-1$, is given by
\begin{align*}
\kappa_g :\alpha\mapsto\Theta(g)\underset{f|g}{\int}\,\psi_f(\alpha).
\end{align*}
From \cite{5}, we have the following about the typical additive character.
 Let $\lambda$ be the canonical additive character of $\mathbb{F}_q$. Thus for $\alpha\in \mathbb{F}_q$ this character is defined as  $\lambda(\alpha)= \exp^{2\pi iTr(\alpha)/p}$, where $Tr(\alpha)$ is the absolute trace of $\alpha$ over $\mathbb{F}_p$.

 Now let $\psi_0$ be the canonical additive character of $\mathbb{F}_{q^m}$, which is simply the lift of $\lambda$ to $\mathbb{F}_{q^m}$, i.e., $\psi_0(\alpha)=\lambda(Tr(\alpha)), \, \alpha\in 
\mathbb{F}_{q^m}$. Now for any $\delta\in \mathbb{F}_{q^m}$, let $\psi_\delta$ be the character defined by $\psi_\delta(\alpha)=\psi_0(\delta\alpha), \, \alpha\in \mathbb{F}_{q^m}$.
Define the subset $\Delta_g$ of $\mathbb{F}_{q^m}$ as the set of $\delta$ for which $\psi_\delta$ has $\mathbb{F}_{q}$-order $g$. So we may also write $\psi_{\delta_g}$ for $\psi_\delta$, where $\delta_g\in \Delta_g$. So with the help of this we can express any typical additive character $\psi_g$ in terms of $\psi_{\delta_g}$ and further we can express this in terms of canonical additive character $\psi_0$.  

In the following sections we will encounter various character sums and a calculation, or at least an estimation, for them will be necessary. The following lemmas are well-established and provide such results.

\begin{lemma}[\cite{10}, Theorem 5.4 - Orthogonality relations] \label{charbound0} 
For any nontrivial character $\chi$ of a finite abelian group $G$ and any nontrivial element $\alpha\in G$, the following hold:
$$\underset{\alpha\in G}{\sum}\chi(\alpha)=0 \quad  \mbox{and} \quad \underset{\chi\in \widehat{G}}{\sum}\chi(\alpha)=0.$$ 
\end{lemma}.

\begin{lemma}[\cite{6}, Corollary 2.3] \label{charbound1}
Take two nontrivial multiplicative characters $\chi_1,\chi_2$ of $\mathbb{F}_{q^m}$. Let $f_1(x)$ and $f_2(x)$ be two monic co-prime polynomials in $\mathbb{F}_{q^m}[x]$, such that none of $f_i(x)$ is of the form $g(x)^{ord(\chi_i)}$ for $i=1,2$; where $g(x)\in \mathbb{F}_{q^m}[x]$ with degree at least 1. Then 
$$\Big|\underset{\alpha\in \mathbb{F}_{q^m}}{\sum}\chi_1(f_1(\alpha))\chi_2(f_2(\alpha))\Big|\leq (n_1+n_2-1)q^{m/2},$$
 where $n_1$ and $n_2$ are the degrees of largest square free divisors of $f_1$ and $f_2$ respectively. 
\end{lemma}


\begin{lemma}[\cite{7}] \label{charbound2}
Let $\chi$ and $\psi$ be two non-trivial multiplicative and additive characters of the field $\Fm$ respectively. Let $\mathfrak{F, G}$ be rational functions in $\Fm (x)$, where $\mathfrak{F}\neq \beta\mathfrak{H}^n$ and $\mathfrak{G}\neq \mathfrak{H}^p-\mathfrak{H}+\beta$, for any $\mathfrak{H}\in \Fm(x)$ and any $\beta\in \Fm$, and $n$ is the order of $\chi$. 
 Then
 $$\left | \underset{\alpha\in\mathbb{F}_{q^m}\setminus \mathbb{S}}{\sum} \chi(\mathfrak{F}(\alpha))\psi(\mathfrak{G}(\alpha))\right|\leq [deg(\mathfrak{G}_\infty)\, +\,  k_0 \,+\, k_1\,-\,k_2\,-\,2 ]q^{m/2},$$ 
 where $\mathbb{S}$ denotes the set of all poles of $\mathfrak{F}$ and $\mathfrak{G}$, $\mathfrak{G}_\infty$ denotes the pole divisor of $\mathfrak{G}$, $k_0$ denotes  the number of distinct zeroes and poles of $\mathfrak{F}$ in the algebraic closure $\overline{\Fm}$ of $\Fm$, $k_1$ denotes the number of distinct poles of $\mathfrak{G}$ (including infinite pole) and $k_2$ denotes the number of finite poles of $\mathfrak{F}$, that are also zeroes or poles of $\mathfrak{G}$.
\end{lemma}

\begin{lemma}[\cite{7}] \label{charbound3} 
Let $f_1(x),\, f_2(x),\, \ldots,\, f_s(x)\in \Fm[x]$ be distinct irreducible polynomials. Let $\chi_1,\, \chi_2,\, \ldots, \, \chi_s$ be multiplicative characters and $\psi$ be a non trivial additive character of  $\Fm$, then
$$ \left| \underset{y\in \Fm \\ f_t(y)\neq 0}{\sum} \chi_1(f_1(y))\chi_2(f_2(y))\ldots \chi_s(f_s(y))\psi(y)\right| \leq kq^{m/2},   $$
where $k=\overset{s}{\underset{i=1}{\sum}}deg (f_i)$.
\end{lemma}

\begin{definition}
For a non-trivial additive character $\psi$ of  the finite field $\Fm$, the sum
$$K(\psi; a, b):= \underset{\alpha\in\Fm^*}{\sum}\psi(a\alpha+b\alpha^{-1}),$$
where $a, b\in \Fm$ is called a \emph{Kloosterman sum}.

\end{definition}

\begin{lemma}[\cite{12}, Theorem 5.45] \label{klo}
If the finite field $\Fm$  has a non-trivial additive character $\psi$ and $a, b\in \Fm$ are not both zero, then the Kloosterman sum satisfies 
$$ \left| K (\psi; a, b) \right| \leq 2 q^{m/2}.$$

\end{lemma}

\section{A lower bound for $\mathfrak{M}(e_1,e_2,g_1,g_2)$} \label{sec3}

  In this section, we try to estimate the number of the elements $\alpha\in \Fm$ such that both  $\alpha$ and $f(\alpha)$ are simultaneously primitive normal elements in $\mathbb{F}_{q^m}$ over $\F$. We consider $q$ an even prime power, i.e. $q=2^k$, where $k$ is a positive integer. Take $e_1,e_2$ such that $e_1, e_2|q^m-1$ and $g_1,g_2$ such that $g_1,g_2|x^m-1$.
Let $\mathfrak{M}(e_1,e_2,g_1,g_2)$ be the number of $\alpha\in\mathbb{F}_{q^m}$, such that $\alpha$ is both $e_1$-free and $g_1$-free and $f(\alpha)$ is  $e_2$-free, $g_2$-free; where $f(x)=\frac{ax^2+bx+c}{dx+e}$ and the matrix $M = \begin{pmatrix}a \, &b, &c\\
0 \, &d\, &e
\end{pmatrix}$ $\in M_{2\times 3}(\Fm)$ is of rank 2. In particular, for our purposes, it suffices to prove that $\mathfrak{M}(e_1,e_2,g_1,g_2)>0$.

For convenience, we use the notations $\omega(n)$ and $g_d$ to denote number of prime divisors of $n$ and the number of monic irreducible factors of $g$ over $\mathbb{F}_{q}$ respectively. Furthermore, we write $W(n):=2^{\omega(n)}$  and $\Omega(g):=2^{g_d}$.

\begin{theorem}
Let $f(x)=\frac{ax^2+bx+c}{dx+e}\in \Fm(x)$ such that the matrix $M = \begin{pmatrix}a \, &b \, &c\\
0 \, &d\, &e
\end{pmatrix}$ $\in M_{2\times 3}(\Fm)$ is of rank 2 and $f(x)\neq yx, \, yx^2$ for any $y\in \Fm$.  Suppose $e_1, e_2$ divide $q^m-1$ and $g_1,g_2$ divide $x^m-1$. If  $M \neq \begin{pmatrix}1 \, &1\,  &0\\
0 \, &1\, &0
\end{pmatrix}$ when $q=2$, $m$ is odd and 
\begin{equation}\label{cond1}
q^{\frac{m}{2}} > 4W(e_1)W(e_2)\Omega(g_1)\Omega(g_2),
\end{equation}
then $\M(e_1,e_2,g_1,g_2) > 0$.

In particular, if
\begin{equation} \label{cond}
q^{m/2}>4W(q^m-1)^2\Omega(x^m-1)^2,
\end{equation}
then $\mathfrak{M}(q^m-1,q^m-1,x^m-1,x^m-1)>0$.
\end{theorem}

\begin{proof} First we establish the result for $d\neq 0$.
From the definition we have, 
\begin{equation}  \label{N}
\mathfrak{M}(e_1,e_2,g_1,g_2)= \theta(e_1)\theta(e_2)\Theta(g_1)\Theta(g_2)\underset{\underset{d_2|e_2}{d_1|e_1}}{\int}\,\underset{\underset{h_2|g_2}{h_1|g_1}}{\int}\,S(\chi_{d_1},\chi_{d_2},\psi_{h_1},\psi_{h_2}),
\end{equation}
where
\begin{align*}
S(\chi_{d_1},\chi_{d_2},\psi_{h_1},\psi_{h_2})=& \underset{\alpha\in\mathbb{F}_{q^m}}{\sum}\chi_{d_1}(\alpha)\chi_{d_2}(f(\alpha))\psi_{h_1}(\alpha)\psi_{h_2}(f(\alpha))
\end{align*}
As there exists some $l_1,\,l_2\in \lbrace 0,\, 1,\, \ldots,\, q^m-2\rbrace$, such that $\chi_{l_i}(\alpha)=\chi_{q^m-1}(\alpha^{l_i})$, for $i=1,\, 2$ and $\psi_{h_i}(\alpha)= \psi_{x^m-1}(\beta_i\alpha)$, for some $\beta_i\in\Fm$ for $i=1,\,2$, we have the following expression:
\begin{align*}
S(\chi_{d_1},\chi_{d_2},\psi_{h_1},\psi_{h_2})=& \underset{\alpha\in\mathbb{F}_{q^m}}{\sum}\chi_{q^m-1}(\alpha^{l_1}\,(f(\alpha))^{l_2})\psi_{x^m-1}((\beta_1\alpha)+\beta_2f(\alpha))\\
=& \underset{\alpha\in\Fm}{\sum}\chi_{q^m-1}({\mathfrak{F}(\alpha))\psi_{x^m-1}(\mathfrak{G}(\alpha)),}
\end{align*}
where $\mathfrak{F}(x)=x^{l_1}(\frac{ax^2+bx+c}{dx+e})^{l_2}$ and $\mathfrak{G}(x)=\beta_1x+\beta_2(\frac{ax^2+bx+c}{dx+e})$, for some  $l_1,\,l_2\in \lbrace 0,\, 1,\, \ldots,\, q^m-2\rbrace$ and $\beta_1,\, \beta_2\in\Fm$.  
 
If $\mathfrak{F}\neq\beta\mathfrak{H}^{q^m-1}$ and $\mathfrak{G}\neq \mathfrak{H}^p-\mathfrak{H}+\beta$, for any $\mathfrak{H}\in\Fm(x)$ and $\beta\in\Fm$, then Lemma~\ref{charbound2} implies 
\begin{equation}
\left| S(\chi_{d_1},\chi_{d_2},\psi_{h_1},\psi_{h_2})\right|\leq 4q^{m/2},
\end{equation}
unless all the four characters are trivial.

Now we consider the case $\mathfrak{F}=\beta\mathfrak{H}^{q^m-1}$, for some $\mathfrak{H}\in \Fm(x)$ and $\beta\in\Fm$. Then $\mathfrak{H}=\frac{\mathfrak{H}_1}{\mathfrak{H}_2}$, for some $\mathfrak{H}_1, \, \mathfrak{H}_2$ coprime polynomials over $\Fm$.
It follows that $x^{l_1}(ax^2+bx+c)^{l_2}\mathfrak{H}^{q^m-1}_2= \beta(dx+e)^{l_2}\mathfrak{H}^{q^m-1}_1,$ and this implies $\mathfrak{H}^{q^m-1}_2|(dx+e)^{l_2}$, hence $\mathfrak{H}_2$ is constant.
Then comparing the degrees of both sides we have $l_1\, + \, 2l_2\,= l_2\,+ \, k_1(q^m-1)$, where $k_1$ is the degree of $\mathfrak{H}_1$ and this gives $l_1\,= 0$ or $1$ i.e. $\mathfrak{H}_1(x)= a^\prime x+b^\prime$.
When $k_1=1$ then $l_1$ must be non-zero, otherwise $l_2=q^m-1$, a contradiction.
Now,
\begin{equation}
(ax^2+bx+c)^{l_2}\,=\, \beta (dx+e)^{l_2}\mathfrak{B}^{q^m-1}x^{q^m-1-l_1}, \label{equn3.6}
\end{equation}
where $\mathfrak{B}(x)=\mathfrak{H}_1(x)/x\in \F[x]$, a constant polynomial. Comparing both sides we have $c=0$. After putting this in the equation, this is possible only if  $ \gcd( dx+e,  ax+b)= x + \frac{c}{d}$ and $q^m-1= l_1+l_2$. In this case $f(x)=\frac{a}{d}x$, which is a contradiction.  Hence $k_1=0$ and $l_1=l_2=0$.

Next, let $\beta_1=0$ and $\beta_2\neq 0$. Then, 
\begin{align*}
\left| S(\chi_{d_1}, \, \chi_{d_2}, \psi_{h_1}, \psi_{h_2}) \right| \,=& \, \left| \underset{\alpha\neq -\frac{e}{d}}{\sum}\psi_{x^m-1}\left(\frac{\beta_2(a\alpha^2+b\alpha+c)}{d\alpha+e} \right) \right|\\
=& \left| \underset{y\neq 0}{\sum} \psi_{x^m-1}\left( \frac{\beta_2}{d^2}ay + \left(\frac{\beta_2}{d^2} \right)(e^2-de+cd^2)y^{-1} \right)  \right| . 
\end{align*}
By Lemma~\ref{klo}, we have 
$$\left| S(\chi_{d_1}, \, \chi_{d_2}, \psi_{h_1}, \psi_{h_2}) \right|\leq 2q^{m/2} \, <\, 4q^{m/2}. $$

Similarly, if $\beta_1\neq 0$ and $\beta_2=0$, by applying Lemma~\ref{charbound0}, we have 
$$\left| S(\chi_{d_1}, \, \chi_{d_2}, \psi_{h_1}, \psi_{h_2}) \right| = \left| \underset{\alpha\in\Fm}{\sum}\psi_{x^m-1}(\beta_1 \alpha)\right| \leq \,1 < \, 4q^{m/2}. $$

If both $\beta_1$ and $\beta_2$ are non-zero, then we can proceed as follows:
\begin{align*}
\left| S(\chi_{d_1}, \, \chi_{d_2}, \psi_{h_1}, \psi_{h_2}) \right|=& \left| \underset{\alpha\neq -\frac{e}{d}}{\sum} \psi_{q^m-1}\left( \beta_1\alpha\,+\, \frac{\beta_2(a\alpha^2+b\alpha+c)}{d\alpha+e} \right) \right|\\
=&\left| \underset{y\neq 0}{\sum} \psi_{q^m-1}\left( \left(\frac{\beta_1}{d}+\frac{\beta_2 a}{d^2}\right)y\,+\, \left(\frac{\beta_2 a e^2}{d^2}-\frac{be}{d}+c\right)y^{-1}+ \left(\frac{\beta_2 b}{d}-\frac{\beta_1 e}{d}\right) \right) \right|\\
& =\left| \underset{y\neq 0}{\sum} \psi_{q^m-1}\left( \left(\frac{\beta_1}{d}+\frac{\beta_2 a}{d^2}\right)y\,+\, \left(\frac{\beta_2 a e^2}{d^2}-\frac{be}{d}+c\right)y^{-1} \right) \right|
\end{align*}
Lemma~\ref{klo} yields
$$\left| S(\chi_{d_1}, \, \chi_{d_2}, \psi_{h_1}, \psi_{h_2} )\right|\leq 2q^{m/2}\, <\, 4q^{m/2}. $$ 

If $\mathfrak{G}= \mathfrak{H}^p-\mathfrak{H}+\beta$ for some $\mathfrak{H}\in\Fm(x)$ and for some $\beta\in\Fm$, then we write $\mathfrak{H}=\frac{H_1}{H_2}$, where $H_1$ and $H_2$ are co-prime polynomials. Continuing this, we have the following.
$$ \frac{\beta_1 x(dx+e)+ \beta_2 (ax^2+bx+c)}{dx+e}=\frac{H^p_1-H_1H^{p-1}_2+\beta H^p_2}{H^p_2}.$$
Immediately from the restriction on the rational polynomial $\frac{ax^2+bx+c}{dx+e}$ we get $(dx+e)$ is co-prime to $\beta_1 x(dx+e)+\beta_2 (ax^2+bx+c)$ and hence $H^p_2$ is co-prime to $H^p_1-H_1H^{p-1}_2+\beta H^p_2$. Then $dx+e= H^p_2$, which is a contradiction as $d\neq0$. It follows that $\mathfrak{G}=0$, i.e. $\beta_1=\beta_2=0$.

Additionally if at least one of $l_1,\, l_2$ is non-zero, then $x^{l_1}(ax^2+bx+c)^{l_2}(dx+e)^{q^m-1-l_2}$ has at most 4 distinct roots and is not of the form $\beta\mathfrak{H}^{q^m-1}$, for $\mathfrak{H}\in\Fm(x)$ and $\beta\in \Fm$. 
Then from Equation~($\ref{equn3.6}$) we have 
\begin{equation*}
S(\chi_{d_1}, \chi_{d_2},\psi_{h_1},\psi_{h_2})=\underset{\alpha\neq-\frac{e}{d}}{\sum}\chi_{q^m-1}\left( \alpha^{l_1}(a\alpha^2+b\alpha+c)^{l_2}(dx+e)^{q^m-1-l_2} \right).
\end{equation*}
From Lemma~\ref{klo} we have the bound $ \left|S(\chi_{d_1}, \chi_{d_2},\psi_{h_1},\psi_{h_2}) \right| \leq 2q^{m/2} <4q^{m/2}.$

In all of the above cases, Equation~(\ref{N}) gives $\M (e_1,e_2,g_1,g_2)>0$ if 
\[ q^{m}>1+4q^{m/2}\left(W(e_1)W(e_2)\Omega(g_1)\Omega(g_2)-1\right), \] hence a sufficient condition is \eqref{cond1}. This concludes the $d\neq 0$ case.

Next, we deal with the case $d=0$. Then $f(x)= \frac{ax^2+bx+c}{e}= \frac{a}{e}x^2+\frac{b}{e}x+\frac{c}{e}= a_1x^2+b_1+c_1$ and
\begin{equation}
\M(e_1,e_2,g_1,g_2) = \theta(e_1)\theta(e_2)\Theta(g_1)\Theta(g_2)\underset{\underset{d_2|e_2}{d_1|e_1}}{\int}\underset{\underset{h_2|g_2}{h_1|g_1}}{\int}S(\chi_{d_1}, \chi_{d_2},\psi_{h_1},\psi_{h_2}).\label{equn3.8}
\end{equation}
Where
\begin{align*}
S(\chi_{d_1}, \chi_{d_2},\psi_{h_1},\psi_{h_2})=&\underset{\alpha\in\Fm}{\sum}\chi_{d_1}(\alpha)\chi_{d_2}(f(\alpha))\psi_{h_1}(\alpha)\psi_{h_2}(f(\alpha))\\ =& \underset{\alpha\in\Fm}{\sum}\chi_{d_1}(\alpha)\chi_{d_2}(f(\alpha))\psi_{h_1}(\alpha)\psi^\prime_{h_2}(\alpha)\\
=& \underset{\alpha\in\Fm}{\sum}\chi_{d_1}(\alpha)\chi_{d_2}(f(\alpha))(\psi_{h_1}\psi^\prime_{h_2})(\alpha),
\end{align*} 
and $\psi^\prime_{h_2}(x)=\psi_{h_2}(f(x))$ for all $x\in\Fm$. 

Now, if $(\chi_{d_1},\chi_{d_2},(\psi_{h_1}\psi^\prime_{h_2})=\psi_h)\neq(\chi_0,\chi_0,\psi_0)$, then we consider following cases. 
\begin{itemize}
\item
If $\psi_{h_1}\psi^\prime_{h_2}=\psi_h$ is non trivial character, then applying Lemma~\ref{charbound3} we have
 $$\left| S(\chi_{d_1},\chi_{d_2},\psi_{h_1},\psi_{h_2})\right|\, = \,\left| S(\chi_{d_1},\chi_{d_2},\psi_{h})\right|\, \leq 3q^{m/2} < 4q^{m/2}.$$
 \item
If $\psi_{h_1}\psi^\prime_{h_2}=\psi_h$ is the trivial character $\psi_0$, then following Lemma~\ref{charbound2}, we have      
 $$\left| S(\chi_{d_1},\chi_{d_2},\psi_{h_1},\psi_{h_2})\right|\, = \,\left| S(\chi_{d_1},\chi_{d_2},\psi_{0})\right|\, \leq 2q^{m/2} < 4q^{m/2}.$$
 \item
Finally, if $\chi_{d_1}=\chi_{d_2}=\chi_0$ then  $\left| S(\chi_{d_1},\chi_{d_2},\psi_{h_1},\psi_{h_2})\right|\, = \,\left| S(\chi_{0},\chi_{0},\psi_{h})\right|\, =0.$ 
\end{itemize}
Hence $\left| S(\chi_{d_1},\chi_{d_2},\psi_{h_1},\psi_{h_2})\right|\,< 4q^{m/2}$ if $(\chi_{d_1},\chi_{d_2},\psi_{h})\neq (\chi_{0},\chi_{0},\psi_{0})$, where $\psi_h=\psi_{h_1}\psi^\prime_{h_2}$. Then, from Equation~(\ref{equn3.8}) we have that a sufficient condition for $\M (e_1,e_2,g_1,g_2)>0$ is given by \eqref{cond1}.

 In particular setting $e_1=e_2=q^m-1$ and $g_1=g_2=x^m-1$, we obtain the sufficient condition \eqref{cond}.
 
  Finally, we briefly consider  the case $c_1=0$ i.e. $c=0$. Then $f(x)= a_1x^2+b_1x= x(a_1x+b_1)$, where $a_1,b_1\in \mathbb{F}_{q^m}$ with $b_1\neq 0$. This time we have
$$\mathfrak{M}(e_1,e_2,g_1,g_2)= \theta(e_1)\theta(e_2)\Theta(g_1)\Theta(g_2)\underset{\underset{d_2|e_2}{d_1|e_1}}{\int}\,\underset{\underset{h_1|g_1}{h_2|g_2}}{\int}\,S(\chi_{d_1},\chi_{d_2},\psi_{h_1},\psi_{h_2}),$$
where $$S(\chi_{d_1},\chi_{d_2},\psi_{h_1}, \psi_{h_2})= \underset{\alpha\in\mathbb{F}_{q^m}}{\sum}\chi_{d_1}(\alpha)\chi_{d_2}(\alpha(a_1\alpha+b_1))\psi_h(\alpha)=  \underset{\alpha\in\mathbb{F}_{q^m}}{\sum}\chi_{d_3}(\alpha)\chi_{d_2}(a_1\alpha+b_1)\psi_h(\alpha).$$
with $\chi_{d_3}=\chi_{d_1}\chi_{d_2}$.  Now, from Lemma~\ref{charbound3}.
  $$\left|S(\chi_{d_1},\chi_{d_2},\psi_{h_1},\psi_{h_2})\right|= \left|  \underset{\alpha\in\mathbb{F}_{q^m}}{\sum}\chi_{d_3}(\alpha)\chi_{d_2}(a_1\alpha+b_1)\psi_h(\alpha)\right|\leq 2q^{m/2} < 4q^{m/2}$$  and the conditions (\ref{cond1}) and (\ref{cond}) follow as before.
\end{proof}

\paragraph{Note:} If $q=2$ and $m$ is odd then for the matrix $M= \begin{pmatrix} 1  &1 &0\\
0  &1 &0 
\end{pmatrix}$, the elements  $\alpha$ and $f(\alpha)$ are not simultaneously normal. Again, for $m$ odd this case is trivially true. Hence we exclude this matrix from our claim.

In the next section, we apply the results on primes dividing $q^m-1$ and irreducible polynomials dividing $x^m-1$ for enhanced results. As already mentioned, this technique was first  introduced by Cohen and Huczynska in \cite{3,4}.

\section{The prime sieve technique} \label{sec4}
We begin this section with the sieving inequality,  as established by Kapetanakis in \cite{8}, which we adjust properly.

\begin{lemma}[Sieving Inequality] \label{sieveineq}

Let $d$ be a divisor of $q^m-1$ and $p_1, p_2,\dots , p_n$ be the remaining distinct primes dividing $q^m-1$. Furthermore, let $g$ be a divisor of $x^m-1$ such that $g_1, g_2, \dots , g_k$ are the remaining distinct irreducible factors of $x^m-1$. Abbreviate  $ \mathfrak{M}(q^m-1, q^m-1, x^m-1, x^m-1) $ to $\mathfrak{M}$. Then
\begin{multline} \label{ineq}
 \mathfrak{M} \geq \overset{n}{\underset{i=1}{\sum}}\mathfrak{M}(p_i d, d, g, g)+  \overset{n}{\underset{i=1}{\sum}}\mathfrak{M}( d,p_i d, g, g)+ \overset{k}{\underset{i=1}{\sum}}\mathfrak{M}(d, d,g_i g, g)\\ 
 +  \overset{k}{\underset{i=1}{\sum}}\mathfrak{M}(d, d, g,g_i g) -(2n+2k-1)\mathfrak{M}(d, d, g, g).
\end{multline}
\end{lemma}

\begin{theorem}
With the assumptions of Lemma~\ref{sieveineq},  define 
$$ \vt := 1 - 2  \overset{n}{\underset{i=1}{\sum}}\frac{1}{p_i} - 2  \overset{k}{\underset{i=1}{\sum}} \frac{1}{q^{{\mathrm deg}(g_i)}}$$
 and  
 $$ \s  := \frac{2n+2k-1}{\vt}+2.$$
 Suppose  $\vt>0$. Then a sufficient condition for the existence of an element $\alpha\in\Fm$ such that both $\alpha$ and $f(\alpha)=\frac{a\alpha^2+b\alpha+c}{d\alpha+e}$ are simultaneously primitive normal over $\mathbb{F}_{q^m}$, where the matrix $M= \begin{pmatrix}
 a &b &c\\
 0 &d &e
 \end{pmatrix}$ is of rank 2 and if $(q,m)=(2, \text{odd})$ then $M \neq \begin{pmatrix}
 1 &1 &0\\
 0 &1 &0
 \end{pmatrix}$ is 

\begin{equation}\label{cond3}
  q^{m/2}> 4 {W(d)}^2\Omega(g)^2\s .
  \end{equation} 
\end{theorem}

\begin{proof}   A key step is to write (\ref{ineq}) in the equivalent form
\begin{multline}\label{altineq}
\M  \geq\sum_{i=1}^n\left( \M(p_id,d,g,g)-\left(1-\frac{1}{p_i}\right) \M(d,d,g,g) \right)\\ +  \sum_{i=1}^n\left(\M(d,dp_i,g,g)-\left(1-\frac{1}{p_i}\right)\M(d,d,g,g)\right)\\
+\sum_{i=1}^k \left(\M(d,d,g_ig,g) -\left(1-\frac{1}{q^{{\mathrm deg}(g_i)}}\right)\M(d,d,g,g)\right)\\
+\sum_{i=1}^k \left(\M(d,d,g,g_ig) -\left(1-\frac{1}{q^{{\mathrm deg}(g_i)}}\right)\M(d,d,g,g)\right) +  \vt\M(d,d,g,g).
\end{multline}

On the right side of (\ref{altineq}), since $\vt >0$,  we can bound the last  term below using (\ref{cond1}).   Thus
\begin{equation} \label{Cond}
 \vt\M(d,d,g,g) \geq \vt \theta^2(d)\Theta^2(g)q^{\frac{m}{2}}(q^{\frac{m}{2}} -4W^2(d)\Omega^2(g)).
 \end{equation}

Moreover, since $ \theta(p_id) =\theta(p_i)\theta(d) =\left(1-\frac{1}{p_i}\right)$ and $\Theta(g_ig)=\Theta(g_i)\Theta(g)=\left(1-\frac{1}{q^{{\mathrm deg}(g_i)}}\right)$  we have from (\ref{N}),
$$ \M(p_id,d,g,g)-\left(1-\frac{1}{p_i}\right) \M(d,d,g,g) =\left(1-\frac{1}{p_i}\right)\theta^2\Theta^2\int_{\substack{d_1|d\\ d_2|d}}\,\int_{\substack{h_1|g\\h_2|g}}S(\chi_{p_id_1}, \chi_{d_2}, \psi_{h_1}, \psi_{h_2}).$$
and $$ \M(d,d,g_ig,g)-\left(1-\frac{1}{q^{{\mathrm deg}(g_i)}}\right) \M(d,d,g,g) =\left(1-\frac{1}{q^{{\mathrm deg}(g_i)}}\right)\theta^2\Theta^2\int_{\substack{d_1|d\\ d_2|d}}\int_{\substack{h_1|g\\h_2|g}}S(\chi_{d_1}, \chi_{d_2}, \psi_{g_ih_1}, \psi_{h_2}).$$
Hence, as for (\ref{cond1}),
\begin{eqnarray}\label{diff1}
\left| \M(p_id,d,g,g)-\left(1-\frac{1}{p_i}\right) \M(d,d,g,g)\right| & \leq & 4\left(1-\frac{1}{p_i}\right)\theta^2(d)\Theta^2(g)\big{(}W(p_id)-W(p_i)\big{)}W(d) \nonumber\\
&=& 4\left(1-\frac{1}{p_i}\right)\theta^2(d)W^2(d) .
\end{eqnarray}
\begin{eqnarray}\label{diff1.1}
\left| \M(d,d,g_ig,g)-\left(1-\frac{1}{q^{{\mathrm deg}(g_i)}}\right) \M(d,d,g,g)\right| & \leq & 4\left(1-\frac{1}{p_i}\right)\theta^2(d)\Theta^2(g)\big{(}\Omega(g_ig)-\Omega(g_i)\big{)}\Omega(g) \nonumber\\
&=& 4\left(1-\frac{1}{q^{{\mathrm deg}(g_i)}}\right)\Theta^2(g)\Omega^2(g) .
\end{eqnarray}
Similarly,
\begin{equation}\label{diff2}
\left| \M(d,d,g,g)-\left(1-\frac{1}{p_i}\right) \M(d,p_id,g,g)\right| \leq 4\left(1-\frac{1}{p_i}\right)\theta^2(d)\Theta^2(g)W^2(d)
\end{equation}
and
\begin{equation} \label{diff3}
 \left|\M(d,d,g_ig,g) -\left(1-\frac{1}{q^{{\mathrm deg}(g_i)}}\right)\M(d,d,g,g)\right| \leq 4 \theta^2(d) \left(1- \frac{1}{q^{\mathrm{deg}(g_i)}} \right)\Theta^2(g)\Omega^2(g).
\end{equation}

Inserting (\ref{Cond}), (\ref{diff1}), (\ref{diff2}) and (\ref{diff3}) in (\ref{altineq}) and cancelling   the common factor $\theta^2(d)\Theta^2(g)$, we obtain (\ref{cond3}) as a condition for $\M$ to be positive (since       $\vt$ is positive). This completes the proof.
 \end{proof}

 We conclude our paper by discussing all the possible cases for fields of characteristic 2.
 
\section{Some estimations for fields of even characteristic} \label{sec5}
The prime purpose  of this section is to analyse the conditions (\ref{cond}) and (\ref{cond3}) for the existence of elements of desired properties in fields of even characteristic. Towards that, we express the pairs $(q,m)$ with the desired properties with extending and developing the techniques employed in \cite{5}, \cite{6} and \cite{7} by the functions presented earlier, leading us to character sums. We have already defined  such pairs $(q,m)$ as \emph{primitive normal} pair.

Also, it is worth mentioning that due to the complexity of the character sums and their fragile behaviour on fields of different orders, it is necessary to distinguish a few cases depending on the order of the prime subfield.  Henceforth, we assume that $q=2^k$, where $k$ is a positive integer.

From now on we use the concept of the radical of $m$ i.e. $m^\prime$ and the radical of $x^m-1$ which is $x^{m^\prime}-1$.  Where $m^\prime$ is such that $m= 2^k m^\prime$, where  $\gcd(2,m^\prime)=1$ and $k$ is a non-negative integer. In fact, when $m'=1$, trivially $k$ is positive.
 
   We split our computations in two cases:
 \begin{itemize}
 \item $m^\prime|q-1$
 \item $m^\prime \nmid q-1$
 \end{itemize}
 
   
Notice that, in the former case, $x^{m^\prime}-1$ splits at most into a product of $m^{\prime}$ linear factors over $\F$. The following result is inspired from Lemma~6.1 of Cohen's work \cite{11}.

\begin{lemma}\label{Lambdaeq}
 For $q=2^k$, where $k\geq 1$, let $d=q^m-1$ and let $g|x^m-1$ with $g_1, g_2, \ldots, g_r$ be the remaining distinct  irreducible polynomials dividing $x^m-1$.
 Furthermore, let us write $ \vt := 1 - \overset{r}{\underset{i=1}{\sum}} \frac{1}{q^{deg(g_i)}}$ and  $ \s := \frac{r-1}{\vt}+2$, with $\vt>0$. 
 Let $m= m^\prime\, 2^k$, where $k$ is a non-negative integer and $\gcd(m^\prime,2)=1$. If $m^\prime| q-1$, then $$\s = \frac{2q^2-6q+aq+4}{aq-2q+2} , $$ 
 where $m^\prime= \frac{q-1}{a}$.  In particular, $\s<2q^2$.
\end{lemma}

We also need the following. We use this result in the next case and all the subsequent cases, unless stated otherwise. 

\begin{lemma}[Lemma~6.2, \cite{11}] \label{Wbound}
For any odd positive integer $n$, $W(n)< 6.46\, n^{1/5}$, where $W$ has same meaning as stated earlier. 
\end{lemma}

From Theorem \ref{cond3} it is clear that some concepts regarding the factorization of $x^m-1$ can be used in order to effectively use the results of the previous section. Such as if $m^\prime|q-1$, then $x^{m^\prime}-1$ splits into $m^\prime$ distinct linear polynomials. Throughout this section we use prime sieve technique result to establish the rest.

\begin{lemma}\label{prop0} For $f(x)\in \Fm(x)$, such that $f(x)=x$ or $f(x)=x^2$, we have $\M (q^m-1, q^m-1, x^m-1,x^m-1) >0$.
\end{lemma}

\begin{proof} The proof  follows from Lemma~4.1 of \cite{2}.
Since $q$ is even, $q^m-1$ is odd, hence both $\alpha$ and $f(\alpha)$ are simultaneously primitive. Similarly, since $m^\prime$ is odd, $\alpha$ and $f(\alpha)$ are simultaneously normal.
\end{proof}

\subsection{Proof of Theorem~\ref{main1}}

Taking $g=1$ in  Inequality~(\ref{cond3})  and applying Lemma~\ref{Wbound}, we have the sufficient condition
$$ q^{\frac{m}{10}}>334\, q^2 . $$
Then for $m^\prime=q-1$, the inequality transforms to
 $$ q^{\frac{q-1}{10}-2}>334,$$
 which holds for $q\geq 64$.

Next, we consider $q=32$ and $m= m^\prime=q-1=31$. Then, by factorizing,  $\omega(27^{26}-1)=12$ and  the pair $(q,m)=(32,31)$ satisfies the condition (\ref{cond3}).  Hence $\mathbb{F}_{32^{31}}$ contains an element $\alpha$ such that both $\alpha$ and $f(\alpha)$ are simultaneously primitive normal with the given conditions.

 In order to reduce our  calculations, we now consider the range  $19\leq m^\prime < \frac{q-1}{3}$, for $q\geq 64$. Then, by Lemma~\ref{Lambdaeq} we have $\s< (2q+2)$. Hence Inequality~(\ref{cond3}) is satisfied if $q^{\frac{m^\prime-1}{10}}> 167(2q+2)$ and this holds for $m^\prime\geq 19$.

  When $m^\prime=\frac{q-1}{3}$, then $\s\leq q$ and then the condition becomes $q^{\frac{m^\prime-1}{10}-1}> 167$ and this holds for $m^\prime\geq 19$. Since $m^\prime=19 \neq \frac{q-1}{3}$ for any $q=2^k$, we may leave this case.

 Next,  we  investigate all   cases with $m^\prime<19$. In the next part, we set $d=q^m-1$, $g=1$ unless mentioned otherwise.
\begin{description}
\item[Case 1, $m^\prime=1$]
 Then $m= 2^j$. Initially we take $j\geq2$.
To check the condition we take $g=X+1$. In that case $\vt=1$ and $\s=1$.
Then the inequality becomes $$q^{\frac{2^j}{10}}> 334.$$

For $q=2$, the condition holds for  $j \geq 7$. Again for $q=4,\, j\geq 6$; for $8\leq q \leq 32 $, $j\geq 5$; for $64 \leq q \leq 2^{10},\, j\geq 4$; for $2^{11} \leq q \leq 2^{20}, \, j\geq 3$ and for $q\geq 2^{21}$ the condition holds for $j\geq 2$.  So we calculate the rest of the pairs $(q,m)$ by calculating $\omega=\omega(q^m-1)$, i.e., the number of distinct prime divisors of $q^m-1$. Hence it suffices to check that $q^{m/2}> 4\cdot W(q^m-1)^2\cdot  2^2$, where $W(q^m-1)=2^{\omega}$. The pairs $(2,4), (2,8), (2,16), (4,4), (4,8), (8,4), (8,8), (16,4), (32,4), (64,4), (128,4), (512,4)$ do not satisfy the condition. We take taking $g=1$ and appropriate value of $d$ and we apply the sieve  condition \ref{cond3} to verify $(128,4), (512,4)$ as primitive normal pairs and declare the rest as possibly exceptional pairs. 

Now we discuss the case when $m=2$. Then any pair $(q,2)$ is primitive normal pair if and only if it is a primitive pair, i.e., there exists $\alpha$ in $\mathbb{F}_{q^2}$
such that both $\alpha$ and $f(\alpha)$ are simultaneously primitive elements of $\mathbb{F}_{q^2}$. For all $q$ such that $q^2-1$ is a Mersenne prime (the primes which are of the form $2^j-1$ for some positive integer $j$ are called \emph{Mersenne primes}) except $(2,2)$, all the elements of $\mathbb{F}^*_{q^2}$ are primitive except the identity and hence pairs $(q,2)$ are primitive normal pairs. However, $(2,2)$ does not fit into this category as $\mathbb{F}_4\cong\frac{\mathbb{Z}_2[x]}{<x^2+x+1>}$ and the primitive elements of $\mathbb{F}_4$ are roots of $f(x) = x^2+x+1$, i.e., $f(\alpha)=0$ is not primitive when $\alpha$ is primitive.

Next, we employ the sufficient condition $q^{1/5}>668$, which holds for $q\geq 2^{47}$ and for the remaining pairs we use sieve condition (\ref{cond3}) to test the existence of the property. When $d=q^2-1$ and $g=x+1$, the condition holds for all $q=2^k$, where $k=13, 17$ and $k\geq 19$. Again choosing appropriate $d$ as in Table~\ref{table1}, we conclude that among the above pairs; $(2^{11},2), (2^{12}, 2), (2^{14},2), (2^{15},2), (2^{16},2),(2^{18},2)$ are primitive normal pairs and the rest are possible exceptions.

Summing up, we have the following possibly exceptional pairs:
$(2,2)$, $(2,4)$, $(2,8)$, $(2,16)$, $(4,2)$, $(4,4)$, $(4,8)$, $(8,2)$, $(8,4)$, $(8,8)$, $(16,2)$,
$(16,4)$, $(32,2)$, $(64,2)$, $(64,4)$, $(128,2)$, $(256,2)$, $(512,2)$ and $(1024,2)$.

\item[Case 2, $m^\prime=3$]
 In this case, $m$ is of the form $m=3\cdot 2^j$, where $j$ is a positive integer and $q= 2^{2k}$ for some $k\geq1$. 
 For $q=4$, take $g=x^{m^\prime}-1$ so that $\s=1 $ and the sufficient condition is $ 4^{\frac{3.2^{j}}{10}}> 167\times (2^3)^2$ , which holds for $j\geq 5$. Hence the pairs under the above condition are primitive normal pairs except $(4,3), (4,6), (4,12), (4,24)$ and $(4,48)$. After employing the sieving condition (\ref{cond3}), see Table~\ref{table1}, we conclude that $(4,24), (4,48)$ are primitive normal pairs and  $(4,3), (4,6), (4,12)$ are possible exceptional pairs.
 
 Then we take  $g=1$. For $q=16$, $\s\leq 22$ the sufficient condition is $ q^{\frac{3.2^j}{10}}> 3672.8$, which holds for $j\geq 4$.
  
 For $q=64, 256$, $\s< 7.51$ and $m^\prime|q-1$, the condition holds for $j\geq 3$. 
 Again for $1024\leq q \leq 2^{16}$, $\s<7.029$ and we need to check
 $ q^{\frac{3.2^j}{10}}> 1251.95$, which holds when $j\geq 2$. 
  For $2^{18}\leq q \leq 2^{34}$, $\s< 7.0001$ and the condition holds for $j\geq 1$; and for $q\geq 2^{35}$ such that $m^\prime| q-1$ the condition holds for $j\geq 0$. 
  
   We calculate the remaining pairs by taking $g=x^3-1$ and using $W(q^m-1)$, $\Omega(x^3-1)$.  So the condition is $q^{m/2}> 4\cdot W(q^m-1)^2\cdot (2^3)^2 $, which all but the pairs $(16,3), (16,6), (64,3),$
   $ (64,6), (256,3), (1024,3),(2^{12},3), (2^{16},3), (2^{20},3) $ fail to satisfy.  Now we choose suitable values of $g$ and $d$ to declare $ (16,12)$, $(16,24)$, $(64,3)$, $(256,3)$, $(1024,3)$, $(2^{12},3)$, $(2^{16},3)$ and $(2^{20},3)$ as primitive normal pairs, as shown in Table~\ref{table1}. 
   
So, we have the following pairs as possible exceptional pairs:      
%
%
$ (4,3)$, $(4,6)$, $(4, 12)$, $(16,3)$ and $(16,6)$.

 From now on assume $m=m'2^j$ with $j \geq0$.
\item[Case 3, $m^\prime=5$]
 Here $m= 5\cdot 2^j$, with non-negative integer $j$. As there are $5$ distinct factors of $x^{m^\prime}-1$, so by calculation we have $\vt>0$ if $q\geq 16$. Then $\s <26$ for $q= 16$ and the sufficient condition is 
$q^{\frac{5.2^j}{10}}> 4340.09$ which holds for $j\geq 3$.

 For $q=256$, $\s\leq 11.7627$, and sufficient condition is $q^{\frac{5.2^j}{10}}> 1963 $. This holds when $j\geq 2$. Again, $4096\leq q \leq 2^{20}$ and $m^\prime| q-1$, the condition is $q^{\frac{5.2^j}{10}}> 1843.57 $ and holds for $j\geq 1$. When $q\geq 2^{21}$ and $m^\prime|q-1$ the condition holds for $j\geq 0$.

 Taking $g=x^5-1$, we check the remaining pairs for the inequality $q^{m/2}>4\cdot 2^{2\omega}\cdot (2^5)^2$ and have the following as possible exceptional pairs $(16,5), (16,10), (256,5), (2^{12}, 5) $. Then we choose proper $d$ and $g$, and verify condition (\ref{cond3}) and have $(16,10), (256,5), (2^{12},5)$ are primitive normal pairs. Then the pair $(16,5)$ is a possible exception.
%
%
\item[Case 4, $m^\prime=7$] 
 Here $m= 7\cdot 2^j$, with non-negative integer $j$. Let $g=x^{m^\prime}-1$ for $q=8$, then $\vt=1$ and $\s=1$. Then  the sufficient condition is 
$q^{\frac{7.2^j}{10}}> 2736128$ which holds for $j\geq 4$.

 For $q=64$, take $g=1$ and $\s\leq 18.64$, then sufficient condition $q^{\frac{7.2^j}{10}}> 3112.8$ holds for $j\geq 2$. Again, $q=512, 2^{12}, 2^{15}$, $\s< 15.3655$ and the condition holds for $j\geq 1$. For $q\geq 2^{16}$, whenever $m^\prime|q-1$ the condition holds for $j\geq 0$.

 Taking $g=x^7-1$, we check the remaining pairs for the inequality $q^{m/2}>4\cdot 2^{2\omega}\cdot (2^7)^2$. After a calculation, we conclude that all the   pairs  are primitive normal pairs except the pairs $(8,7)$ $(8,14) $.  

\item[Case 5, $m^\prime=9$]
 Here $m= 9\cdot 2^j$, with non-negative integer $j$. As $m^\prime=9$, there are $9$ distinct factors of $x^{m^\prime}-1$. When $g=1$ we have $\vt>0$ if $q\geq 32$. Then $\s <38.2667$ for $q=64$ and the sufficient condition is 
$q^{\frac{9.2^j}{10}}> 6387.72$ which holds for $j\geq 2$.

 For $q =2^{12}$, sufficient condition holds for $j\geq 1$.  When $q\geq 2^{13}$ and $m^\prime|q-1$ the condition holds for $j\geq 0$.

 Taking $g=x^9-1$, we check the remaining pairs for the inequality $q^{m/2}>4\cdot 2^{2\omega}\cdot (2^9)^2$ and take  pair $(64,9), $ which does not  satisfy the inequality. After calculating with suitable values of $d$ and $g$, as shown in Table~\ref{table1}, we conclude that $(64,9)$ is also a primitive normal pair, i.e., all the pairs are primitive normal.

\item[Case 6, $m^\prime=11$]
 Here $m= 11\cdot 2^j$, with non-negative integer $j$. As there are $11$ distinct factors of $x^{m^\prime}-1$, so by calculation for $g=1$ we have $\vt>0$ if $q\geq 32$.

 For $q = 2^{10}$, $\s < 27.3585$ and sufficient condition $q^{\frac{11.2^j}{10}}> 4566.86$  holds for $j\geq 1$.  When $q\geq 2^{11}$ and $m^\prime|q-1$ the condition holds for $j\geq 0$.

 Taking $g=x^{11}-1$, we check the remaining pairs for the inequality $q^{m/2}>4\cdot 2^{2\omega}\cdot (2^{11})^2$, then we have 3 possible exceptional pairs. By calculating with suitable values of $d$ and $g=x^{m^\prime}-1$, we conclude that all the pairs are primitive normal. 

\item[Case 7, $m^\prime=13$]
 Here $m= 13\cdot 2^j$, with non-negative integer $j$. As there are $13$ distinct factors of $x^{m^\prime}-1$, by calculation, we have $\vt>0$ if $q\geq 32$. 
 For $q\geq 64$ and $m^\prime|q-1$, take $g=1$ then $\s < 44.1053$ and the sufficient condition holds for $j\geq 0$. We conclude that all the pairs are primitive normal.

\item[Case 8, $m^\prime=15$]
 Here $m= 15\cdot 2^j$, with non-negative integer $j$. As $m^\prime=15$, there are $15$ distinct factors of $g=x^{m^\prime}-1$. For $q=16$, the sufficient condition for existence of primitive normal element is  $q^{15.2^j/10}> 167\cdot (2^{15})^2$. This condition holds for $j\geq 3$.

 For $q=256$, and $g=1$ we have $\s < 56.882$ and the sufficient condition $q^{\frac{15.2^j}{10}}> 9446.06$ holds for $j\geq 1$.  When $q> 256$ and $m^\prime|q-1$ the condition holds for $j\geq 0$.

 Taking $g=x^{15}-1$, we check the remaining pairs on the inequality $q^{m/2}>4\cdot 2^{2\omega}\cdot (2^{15})^2$ and we obtain $(16,15), (16,30), (256,15)$ as possible exceptional pairs. By calculating with compatible values of $d$, $g$ in the prime sieve condition (\ref{cond3}),  we get that $(16,30), (256,15)$ are primitive normal pairs. Hence we declare $(16,15)$ as an exceptional pair.  

\item[Case 9, $m^\prime=17$]

 Here $m= 17\cdot 2^j$, with non-negative integer $j$. As there are $17$ distinct factors of $x^{m^\prime}-1$, by calculation, for $g=1$, we have $\vt>0$ if $q\geq 64$. 
 
When $q\geq 64$, the sufficient condition is 
$q^{\frac{17.2^j}{10}}> 12085.5$ which holds for $j\geq 0$ whenever $m^\prime|q-1$. Hence  we have that all the pairs of this case are primitive normal.

\end{description}

For each of the individual  pairs $(q,m)$ listed above that do not satisfy the sufficient condition based on Lemma~\ref{Wbound}, we can test them further by means of the sufficient condition (\ref{cond3}) after factorising completely $x^m-1$ and $q^m-1$ and making a choice of polynomial divisor $g$ of $x^m-1$ and factor $d$ of $q^m-1$.  In practice, the best choice is to choose $p_1, \ldots, p_n$ and sometimes, the ``largest'' irreducible factors $g_1, \ldots, g_k$ of $x^m-1$ to ensure that $\vt$ is positive (and not too small).   Here the multiplicative aspect of the sieve is more significant.  Table~\ref{table1} summarizes the pairs in which the test yielded some positive conclusion.
This concludes the proof.

\begin{table}[h]
\begin{center}
\begin{tabular}{|c|c|c|c|c|c|c|c|}
\hline 
$(q,m)$ & $d$ & $n$ & $g$ & $k$ & $\s$ & $q^{m/2}$ & $4 W(d)^2\Omega^2(g)\Lambda$ \\
\hline
(128,4) & 3 & 5 & $x+1$ & 0 & 21.9523 & 16384 &  1404.95 \\ 
\hline
(512,4) & 15 & 6 & $x+1$ & 0 & 32.9531 & 262144 &  8435.99 \\
\hline
$(2^{11},2)$ & 3 & 3 & $x+1$ & 0 & 7.6329 & 2048 &  488.506 \\  
\hline
$(2^{12},2)$ & 15 & 4 & $x+1$ & 0 & 18.1107 & 4096 &  1159.08 \\  
\hline
$(2^{14},2)$ & 3 & 5 & $x+1$ & 0 & 21.9523 & 16384 &  1404.95 \\
\hline
$(2^{15},2)$ & 3 & 5 & $x+1$ & 0 & 22.0596 & 32768 &  1411.81 \\ 
\hline
$(2^{16},2)$ & 3 & 4 & $x+1$ & 0 & 16.7511 & 65536 &  1072.07 \\
\hline
$(2^{18},2)$ & 15 & 6 & $x+1$ & 0 & 32.9531 & 262144 &  8435.99 \\   
\hline
(4,24) & 15 & 7 & $x^3-1$ & 0 & 34.2484 & $1.6777\times 10^7$ &  140283 \\ 
\hline
(4,48) & 15 & 10 & $x^3-1$ & 0 & 50.3795 & $2.81475 \times 10^{14}$ &  206354 \\ 
\hline 
(16,12) & 15 & 7 & $x^3-1$ & 0 & 34.2484 & $1.6777\times 10^7$ &  140283   \\ 
\hline
(16,24) & 15 & 10 & $x^3-1$ & 0 & 50.3795 & $2.81475 \times 10^{14}$ & 206354 \\
\hline  
(64,3) & 3 & 3 & $x+1$ & 2 & 19.3369 & 512 &  309.39 \\ 
\hline
(256,3) & 15 & 4 & $x+1$ & 2 & 28.2612 & 4096 & 452.179 \\
\hline
(1024,3)& 3 & 5 & $x^3-1$ & 0 & 22.0596 & 32768 & 22589 \\
\hline
$(2^{12},3)$ & 15 & 6 & $x^3-1$ & 0 & 32.9531 & 262144 & 134976 \\
\hline
$(2^{16},3)$& 15 & 7 & $x^3-1$ & 0 & 34.2484 & $1.67772 \times 10^{7}$ & 140281 \\
\hline
$(2^{20},3)$& 15 & 9 & $x^3-1$ & 0 & 82.2883 & $1.07374 \times 10^{9}$ & 337053 \\ 
\hline
(16,10) & 3 & 6 & $x^5-1$ & 0 & 60.7588 & $1.04858 \times 10^{6}$ & 995472 \\
\hline 
(256,5) & 3 & 6 & $x^5-1$ & 0 & 60.7588 & $1.04858 \times 10^{6}$ & 995472 \\
\hline  
$(2^{12},5)$ & 15 & 9 & $x^5-1$ & 0 & 87.8157 & $1.07374 \times 10^{9}$ & $5.75509 \times 10^6$ \\
\hline  
(8,28) & 15 & 10 & $x^7-1$ & 0 & 49.0678 & $4.39805 \times 10^{8}$ & $5.14313 \times 10^7$ \\
\hline  
(8,56) & 15 & 15 & $x^7-1$ & 0 & 106.643 & $1.93428 \times 10^{25}$ & $1.11828 \times 10^8$ \\
\hline 
(64,9) & 3 & 5 & $x^9-1$ & 0 & 17.4747 & $1.34218 \times 10^{8}$ & $7.32942 \times 10^7$ \\
\hline
(16,30) & 15 & 13 & $x^{15}-1$ & 0 & 293.517 & $1.15292 \times 10^{18}$ & $2.01703 \times 10^{13}$ \\
\hline  
(256,15) & 15 & 13 & $x^{15}-1$ & 0 & 293.517 & $1.15292 \times 10^{18}$ & $2.01703 \times 10^{13}$ \\
\hline 
\end{tabular}
\caption{Pairs $(q,m)$ appearing in the proof of Theorem~\ref{main1}, in which the corresponding test yielded a positive conclusion.\label{table1}}
\end{center}
\end{table}
%
%
%
%
%
%
%
 
\subsection{Proof of Theorem~\ref{main2}}
For our next main theorem, we need the following well-known facts, see \cite[Theorem~2.47]{10}.
 Let $u$ be the order of $q$ mod $m^\prime$. Then $x^{m^\prime}-1$ is a product of irreducible polynomial factors of degree less than or equal to $u$ in $\mathbb{F}_q[x]$;
in particular, $u\geq 2$ if $m^\prime \nmid q-1$.
 Let $M$ be the number of distinct  irreducible polynomials of $x^m-1$ over $\mathbb{F}_q$ of degree less than $u$. 
   Let $\si(q,m)$ denotes the ratio
 $$ \si(q,m):= \frac{M}{m},$$
   where $m\si(q,m)= m^\prime\si(q,m^\prime)$.

 From Proposition~5.3 of \cite{3}, we deduce the following bounds.

\begin{lemma}\label{rhobound}
Suppose $q=2^k$. Then the  following hold.
\begin{itemize}
\item $\si(2,3)= \frac{1}{3}$; $\si(2,5)=\frac{1}{5}$; $\si(2,9)=\frac{2}{9}$; $\si(2,21)=\frac{4}{21}$ otherwise $\si(2,m)\leq \frac{1}{6}$.
\item $\si(4,9)=\frac{1}{3}$; $\si(4,45)=\frac{11}{45}$; otherwise $\si(4,m)\leq \frac{1}{5}$.
\item $\si(8,3)=\si(8,21)=\frac{1}{3}$; otherwise $\si(8,m)\leq \frac{1}{5}$.
\item If $q\geq 16$, then $\si(q,m)\leq \frac{1}{3}$. 

\end{itemize}
\end{lemma}

 In, to develop suitable sufficient conditions, we need Lemma~7.2 from \cite{11}.

\begin{lemma}
Assume that $q=2^k$ and $m$ is a positive integer such that $m^\prime\nmid q-1$.  Let $u (>1)$ stand for the order of $q \mod m'$.   Let $g$ be the product of the irreducible factors of $x^{m'}-1$ of degree less than $u$.  Then, in the notation of Lemma  $\ref{Lambdaeq}$, we have $\s\leq m^\prime$. \label{lemma5.6}
\end{lemma}

We need few more  conditions, which we can derive from Lemma~4.2 of \cite{9}.

\begin{lemma}
For any $n, \alpha\in \mathbb{N}$, $W(n)\leq b_{\alpha,n} n^{1/\alpha}$, where $b_{\alpha,n}=\frac{2^s}{(p_1 p_2\cdots p_s)^{1/\alpha}}$ and $p_1,\,p_2,\, \ldots,\, p_s$ are the primes $\leq 2^\alpha$ that divide $n$ and $W$ has the same meaning as before.\label{lemma5.7 }
\end{lemma}
From these we immediately derive the lemma below. 
 
\begin{lemma}
For  $n\in \mathbb{N}$ and 
\begin{itemize}
\item[(i)] $\alpha=6$, $W(n)< 37.4683\,n^{1/6}$,
\item[(ii)] $\alpha=8$, $W(n)< 4514.7\, n^{1/8}$,
\item[(iii)] $\alpha=14$, $W(n)< (5.09811\times 10^{67})n^{1/14}$,
\end{itemize}
 where $W$ has the same meaning as earlier.\label{lemma5.8}
\end{lemma}

\begin{lemma}
Let $q=2$, $M\neq \begin{pmatrix}
1 &1 &0\\
0 &1 &0
\end{pmatrix}$ and $m^\prime\nmid q-1$, then there exists an element $\alpha\in \Fm$ such that $\alpha$, $f(\alpha)$ are simultaneously primitive and normal over $\F$, i.e., $(q,m)$ are primitive normal pairs except, possibly, the pairs $(2,3)$, $(2,5)$, $(2,6)$, $(2,7)$, $(2,9)$, $(2,10)$, $(2,11)$, $(2,12)$, $(2,14)$, $(2,15)$, $(2,18)$, $(2,21)$, $(2,24)$, $(2,30)$.
\end{lemma} 

\begin{proof}
First, let $m^\prime =3$. Then $x^\prime-1$ can be factorised into one linear and one quadratic factor. Then the condition becomes $2^{m/10}> 2672$, which holds for $m\geq 114$. Next let $m=96$. Then $\omega=12$ and the condition is $q^{m/2}> 2^{2\omega+6}$, which holds. But the remaining pairs $(2,3),(2,6),(2,12),(2,24),(2,48)$ do not satisfy the above condition. We perform further research on these pairs by taking compatible $d$ and $g$ in the sieve condition (\ref{cond3}) as demonstrated in Table~\ref{table2} and conclude that $ (2,48)$ is primitive normal pair and $(2,3),(2,6),(2,12), (2,24)$ are possible exceptional pairs.

Again, if $m^\prime=5$, then $x^\prime-1$ can be factorised into one linear and one fourth degree polynomial. Then the condition becomes $2^{m/10}>2672$, which holds for $m\geq 114$. Thus, for the remaining pairs a condition is $q^{m/2}>2^{2\omega+6}$ and by calculating $\omega(q^m-1)=\omega$, we have the following exceptional pairs $(2,5),(2,10),(2,20),(2,40)$. Again from Table~\ref{table2} we can conclude that the only possible exceptional pairs are $(2,5), (2,10), (2,20)$.

For $m^\prime=9$, $x^\prime-1$ is a product of one linear, one quadratic and one sextic polynomial and the condition is $2^{m/10}>10688$, which holds for $m\geq 134$. For the remaining the pairs we use the condition $q^{m/2}>2^{2\omega+8}$ and by calculating the value of $\omega$, we have the following exceptional pairs $(2,9),(2,18),(2,36).$ From Table~\ref{table2}, we can conclude that $(2,36)$ is a primitive normal pair and hence final possible exceptional pairs are $(2,9)$ and $(2,18).$

Now, for $m^\prime=21$, $x^\prime-1$ is a product of one linear, one quadratic, two cubic and two distinct sextic polynomials. Then the condition is $2^{m/10}>684032$ and the condition holds for $m\geq 194$. For the remaining pairs we use the condition $q^{m/2}> 2^{2\omega+14}$ and by calculating the value of $\omega(q^m-1)=\omega$, we have the following exceptional pairs $(2,21), (2,42)$, from which we can declare the pair $(2,42)$ as primitive normal pair from Table~\ref{table2}. Hence the ony possible exceptional pair is $(2,21)$.

For the remaining pairs i.e. $q=2$, $m^\prime \nmid q-1$ and $m^\prime \neq 3,\,5,\,9,\,21$, we consider two cases, viz. (i) $m$ is odd and (ii) $m$ is even.  

\textbf{Case (i): $m$ is odd.} We apply Lemma~\ref{lemma5.6} to obtain the condition $q^{m/2}> 4\cdot 2^{2\omega}\cdot 2^{2m\si(q,m)}\cdot m$. Then by Lemmas~\ref{rhobound} and  \ref{lemma5.8}, the condition transforms to $2^{m/42}>1.03991\cdot 10^{136}.m$, which holds for $m\geq 19577$. Let $m\leq 19576$, then $\omega\leq 1620$, and, by applying these on the condition $2^{m/6}>m2^{2\omega+2}$, we conclude that the condition holds for $m\geq 19538$. Maintaining the flow we have that the condition holds for $m\geq 19333.$
 
 For the remaining pairs we calculate the exact value of $\omega$ and able to detect 37 pairs where $m=7$, $11$, $13$, $15$, $17$, $19$, $23$, $25$, $27$, $29$, $31$, $33$, $35$, $37$, $39$, $41$, $43$, $45$, $47$, $51$, $53$, $55$, $57$, $59$, $65$, $67$, $69$, $71$, $73$, $75$, $77$, $79$, $81$, $135$, $165$ and $225$; which don't satisfy the condition. Again for $d=q^m-1$ and  $g=x^{m^\prime}-1$, applying the prime sieve we are able to declare 20 of them as primitive normal pairs. Then by choosing compatible $d$ and $g$ (as shown in Table~\ref{table2}) we are able to determine another 13 pairs $(2,17), (2,19), (2,23), (2,25), (2,27), (2,29), (2,31), (2,33), (2,35)$, $(2,39)$, $(2,45)$, $(2,51)$ as primitive normal pairs. Hence, we conclude that following are the possible exceptional pairs $(2,7),\, (2,11),\, (2,13),\, (2,15)$. 
 
\textbf{Case (ii): $m$ is even.} Once again, we shall break this discussion into two parts.
\begin{itemize}
 \item[$4\mid m$:] Then by Lemma~\ref{lemma5.8}, $W(q^m-1)<37.4683 \, q^{m/6}$ and for $4|m$, $\si(q,m)\leq m/24$. Then to show $\M(q^m-1,q^m-1,x^m-1,x^m-1)>0$ it is sufficient to show $2^{m/12}>5615.49\,m$, which holds for $m\geq 248$. Then we calculate the exact value of $\omega$ and check the condition $2^{5m/12}>2^{2\omega+2}$, for $m\leq 143$ and identify the pairs $(2,28), (2,44), (2,52), (2,56), (2,60)$ which do not satisfy the condition. But from Table~\ref{table2}, we can conclude that all of them are primitive normal pairs. Hence in this particular case all pairs $(q,m)$ are primitive normal pairs.
 
\item[$4\nmid m$:] From Lemma~\ref{lemma5.8}, $W(q^m-1)<4514.7 \, q^{m/8}$ and in this case $\si(q,m)\leq m/12$. Now, a sufficient condition for the existence of a primitive normal pair is $2^{m/12}> 8.153\times 10^7$, which holds for $m\geq 420$. For the remaining pairs we use the prime sieve condition (\ref{cond3}) for $d=q^m-1$ and $g=x^{m^\prime}-1$ and identify the pairs $(2,14), (2,22), (2,30), (2, 70)$ which fail to satisfy the condition. Again, by observing the condition (\ref{cond3}) for appropriate values of $d$ and $g$ we are able to identify the pairs $(2,22), (2,70)$ as primitive normal pairs; the calculations are listed in Table~\ref{table2}. Hence the only possible exceptional pairs are $(2,14)$ and $(2,30)$.
\end{itemize}
The proof is now complete.
\end{proof}

The following lemma is derived from Lemma~\ref{lemma5.7 }.
\begin{lemma}
For $n\in \mathbb{N}$, $W(n)<1.10992\cdot 10^9\, n^{1/10}$ and $W(n)< 4.24455\cdot 10^{14}\,n^{1/11}$. \label{lemma5.10}
\end{lemma}

\begin{lemma}
For $q=4$ and $m'\nmid q-1$, all the pairs $(q,m)$ are primitive normal pairs, except for the possible exceptional pairs $(4,5), (4,7), (4,9), (4,10)$.
\end{lemma}

\begin{proof}
We shall start this discussion with the case $m^\prime=45$. In this case $x^{m^\prime}$ is a product of 3 linear, 6 quadratic, 2 cubic and 4 sextic factors. Let $g$ be the product of the linear factors, then $\vt=0.5927$ and $\s=20.56$. After this, the sufficient condition becomes $4^{m/10}>167 \cdot (2^3)^2\cdot 20.56$, which holds for $m\geq 90$. When $m=45$, then $\omega=\omega(4^m-1)=11$ and the pair $(4,45)$ satisfies the condition $4^{m/2}> 2^{2\omega+8}\cdot 20.56$. Hence $(4,45)$ is also a primitive normal pair.

Now we are heading towards the next case, which is $m^\prime=9$. Then 
$x^{m^\prime}-1$ is a product of 3 linear and 2 cubic factors. Now we take $g$ as the product of three linear factors, then $\vt=0.9375$ and $\s= 5.5$. These yield the condition $4^{m/10}>167\cdot (2^3)^2\cdot 5.5$, which holds for $m\geq 144$. 

For the remaining pairs we verify the sufficient condition $4^{m/2}>2^{2\omega+8}\cdot 5.5$ by calculating the exact value of $\omega$. After this, we can conclude that the pairs $(4,36), (4,72)$ are primitive normal. From Table~\ref{table2},  we conclude that $(4,18)$ is also a primitive normal pair, thus the only possible exceptional pair is $(4,9)$.

Next we have the case $q=4$, $m^\prime\nmid q-1$ and $m^\prime\neq 9, 45$. At first we consider $m$ even.  In this case $\si(q,m)\leq m/10$ and by Lemma~\ref{lemma5.10}, $W(q^m-1)<1.10992\cdot 10^9\,q^{m/10}$. Hence a sufficient condition for our purpose is $4^{m/5}>4.83296\cdot 10^{18}\, m$, which holds for $m\geq 174$. For the remaining pairs we use the condition $4^{2m/5}>2^{2\omega+2}\,m$ and calculate $\omega=\omega(4^m-1)$ explicitly. Among the remaining pairs, $(4,10), (4,14), (4,20), (4,22), (4,28), (4,30)$ do not satisfy the condition. Again for appropriate values of $d$ and $g$, $(4,14), (4,20), (4,22), (4,28), (4,30)$ satisfy the sieve condition as given in Table~\ref{table2}. Hence the only possible exceptional pair is $(4,10)$.

Now, we consider the case $m$ odd. Here $\si(q,m)=1/5$ and from Lemma~\ref{lemma5.10} we have $W(q^m-1)< 4.24455\cdot 10^{14}\, q^{m/11}.$ Then the sufficient condition is $4^{m/11}>7.20647\cdot 10^{29}\,m$, which holds for $m\geq 597$. Afterwards, we use the condition $4^{3m/10}>2^{2\omega+2}\,m$ to test the remaining pairs by calculating the $\omega=\omega(q^m-1)$. The pairs $(4,5)$, $(4,7)$, $(4,11)$, $(4,13)$, $(4,15)$, $(4,25)$, $(4,27)$, $(4,29)$, $(4,33)$, $(4,35)$ and $(4,39)$ do not satisfy the condition. Now we take $d=q^m-1$ and $g=x^m-1$ in the prime sieve condition (\ref{cond3}) and detect $(4, 27), (4,29), (4,33)$ and $(4, 39)$ as primitive normal pairs. Again, by choosing compatible values of $d$ and $g$ in condition (\ref{cond3}) (as shown in Table~\ref{table2}) we conclude that all of the remaining pairs are primitive normal pairs.
This concludes the proof.
%
%
%
\end{proof}

\begin{lemma}Let $q=8$ and $m^\prime\nmid q-1$, then all the pairs $(q,m)$ are primitive normal pairs, unless $(q,m)$ is one of the pairs $(8,3), (8,5)$ and $(8,7)$.
\end{lemma}

\begin{proof}
We begin our discussion with $m^\prime=3$. Then $x^{m^\prime}-1$ is a product of a linear and a quadratic polynomial. If we take $g$ to be the linear polynomial, then $\vt=0.96875$ and $\s<3.04$. It follows that a sufficient condition for the existence of primitive normal pair is $8^{m/10}> 167\cdot 2^2\cdot 3.04$ and this holds for all $m\geq 48$. For the remaining pairs we use the condition $8^{m/2}>2^{2\omega+4}\cdot 3.04$ by explicitly calculating the value of $\omega$. Then the pairs $(8,3), (8,6), (8,12)$ are the ones which fail to satisfy the inequality. By choosing appropriate values of $d$ and $g$ in condition (\ref{cond3}), as shown in Table~\ref{table2}, we conclude that $(8,3)$ is the only possible exceptional pair. 

For the next stage we choose $m^\prime=21$, that is, $x^{m^\prime}-1$ is product of one linear, one quadratic, two cubic and two sextic polynomials. We choose $g$ as the product of the linear and the quadratic factor. Then $\vt=0.992172$ and $\s< 9.06$ which yields the sufficient condition $8^{m/10}>167\cdot (2^4)^2\cdot 9.06$, which holds for $m\geq 84$. Then the condition $8^{m/2}>2^{2\omega+10}\cdot 9.06$ comes into play to detect the primitive normal pairs by taking the exact value of $\omega$. From this, we declare that the remaining pairs $(8,21), (8,42)$ are also primitive normal pairs.

Now, we are heading for the final stage, i.e., $q=8$, $m^\prime\nmid q-1$ and $m^\prime\neq 3, 21$. Form Lemmas~\ref{rhobound} and \ref{lemma5.7 }, we have $\si(q,m) \leq 1/5$ and $W(q^m-1)< 37.4683 q^{m/6}$. It follows that for the existence of primitive normal pairs, a sufficient condition is $8^{m/30}> 5616 m$, which holds for $m\geq 202$.

For the remaining pairs, we use the condition $8^{11m/30}>2^{2\omega+2}m$ by determining the value of $\omega$. For $m\leq 201$, $\omega\leq 85$ this holds for $m\geq 164$. Next we take $m\leq 163$ and then $\omega \leq 72$. For these the condition holds for $m\geq 140$. Now repeating the above process we get that the condition holds for $m\geq 92$ and among the remaining pairs $(8,5), (8,9), (8,10), (8,11), (8,15), (8,20)$ are the ones which fail to satisfy the condition. Then choosing appropriate value of $l$ and $g=x^{m^\prime}-1$ in condition (\ref{cond3}) we are able to declare all but the pair $(8,5)$ as primitive normal pairs. 

Our proof os now complete.
%
%
%
\end{proof}

\begin{lemma}Let $q\geq 16$ and $m^\prime\nmid q-1$, then all the pairs $(q,m)$ are primitive normal pairs, unless $(q,m)=(32,3)$.
\end{lemma}

\begin{proof}
We shall break the discussion into 4 cases (I--IV). Lemma~\ref{rhobound} implies, that   $\vartheta(q,m)\leq \frac{1}{3}$ in all four cases. Furthermore, we take $g$ to be the product of irreducible polynomials dividing $x^m-1$ of degree less than $u$.

\textbf{Case I: $q=16$;} For this case we apply Lemma~\ref{lemma5.8} i.e. $W(q^m-1)< 4514.7 q^{m/8}$. Then to show $\mathfrak{M}(q^m-1,q^m-1,x^m-1,x^m-1)>0$ it is sufficient to show that $16^{m/12}>8.15265\cdot 10^7 m $, which holds for $m\geq 110$. We use the condition $16^{m/2}> 2^{2\omega+2}m$ to test the remaining pairs by plotting value of $\omega$ and conclude that the pairs $(16,7)$, $(16,9)$, $(16,11)$, $(16,13)$, $(16,14)$, $(16,18)$ and $(16,21)$ fail to satisfy the condition. Further, we choose compatible $l$ and $g=x^{m^\prime}-1$ in condition (\ref{cond3}) and conclude that all of them, except $(16,7)$, are primitive normal pairs. Finally, from Table~\ref{table2}, we obtain $(16,7)$ is also a primitive normal pair.

\textbf{Case II: $q=32$;} From Lemma~\ref{lemma5.8} we have $W(q^m-1)< 37.4683 q^{m/2}$ and proceeding as above with the sufficient condition $32^{m/30}>1403.87 m$, which is true for all $m\geq 103$. For rest of the pairs we use the condition $32^{11m/30}>2^{2\omega+2}m$, which proves that all the pairs $(q,m)$ are primitive normal pairs unless $(q,m)$ is one of the pairs $(32,3),(32,5), (32,6),(32,9), (32,10), (32,12)$. Furthermore applying the prime sieve condition (\ref{cond3}) for compatible $l$ and $g=x^{m^\prime}-1$, we confirm that all of them are primitive normal pairs except $(32,3)$.

\textbf{Case III: $q=64$;} Using Lemma~\ref{lemma5.8} we have $W(q^m-1)< 37.4683 q^{m/2}$ and for $\mathfrak{M}(q^m-1,q^m-1,x^m-1,x^m-1)>0$ the sufficient condition is $64^{m/18}>5601.03 m$, which is true for all $m\geq 49$. We use the condition $64^{7m/18}>2^{2\omega+2}m$, to investigate the existence of the property in the rest of the pairs and conclude that all the pairs $(q,m)$ are primitive normal pairs unless $(q,m)$ is $(64,5)$ or $(64,10)$. Later applying the prime sieve condition (\ref{cond3}) for compatible $l$ and $g=x^{m^\prime}-1$, we confirm that all of them are primitive normal pairs.

\textbf{Case IV: $q\geq 128$;} Lemma~\ref{lemma5.8} yields $W(q^m-1)< 37.4683 q^{m/2}$ and for $\mathfrak{M}(q^m-1,q^m-1,x^m-1,x^m-1)>0$ it is sufficient to show that $q^{m/6}>1403.87\cdot 2^{2m/3} m$, which is true for all $q\geq 128$ and $m\geq 18$. We use the condition $q^{m/2}>2^{2\omega+2+2m/3}m$, to test the existence of the property in rest of the pairs (149 in total) and all the pairs $(q,m)$ are primitive normal pairs except $(128,3)$. Then, from Table~\ref{table2}, we confirm that all of them are primitive normal pairs. This concludes our proof.
%
%
%
%
%
%
%
%
%
%
%
%
%
 \end{proof}


\begin{table}
\begin{center}
\begin{tabular}{|c|c|c|c|c|c|c|c|}
\hline 
$(q,m)$ & $d$ & $n$ & $g$ & $k$ & $\Lambda$ &$q^{m/2}$ & $3W(d)^2\Omega(g)\Lambda$ \\ 
\hline 
(2,48) & 105 & 6 & $x+1$ & 1 & 70.8428 &  $1.67772 \time 10^7$   & 72543 \\ 
\hline  
(2,40) & 3 & 6 & 1 & 2 & 82.1256 &     $1.04858 \times 10^6$   &  21031.1 \\ 
\hline
(2,36) & 15 & 6 & $x^9-1$ & 0 & 32.9531 &    262144    & 134976 \\ 
\hline 
(2,42) & 3 & 5 & $x^{21}-1$ & 0 & 15.9379 &  $2.09751 \times 10^6$ & 16320.4\\
\hline 
(2,17) & $q^m-1$ & 0 & $x+1$ & 2 & 5.04762 & 362.039  & 323.048 \\ 
\hline 
(2,19) & $q^m-1$ & 0 & $x+1$ & 1 & 3.00001 & 724.077 & 192.001\\
\hline
(2,23) & 47 & 1 & $x+1$ & 2 & 7.00984 & 2896.31 & 448.63\\
\hline
(2,25) & 31 & 2 & $x+1$ & 2 & 10.0408 & 5792.62 &  642.611\\
\hline
(2,27) & 7 & 2 & $x+1$ & 3 & 22.3926 & 11585.2 &  1433.13\\
\hline
(2,29) & 233 & 2 & $x^{29}-1$ & 0 & 5.00834 & 23170.5 & 1282.14\\
\hline
(2,31) & $q^m-1$ & 0 & $x+1$ & 6 & 19.6 & 46341 & 1254.4\\
\hline
(2,33) & 7 & 2 & $x+1$ & 4 & 35.7918 & 92681.9 & 2290.68\\
\hline
(2,35) & 31 & 3 & $x+1$ & 5 & 47.4422 & 185364 & 3036.3\\
\hline
%
%
%
%
%
%
%
%
%
%
(2,39) & 7 & 3 & $(x+1)(x^2+x+1)$ & 3 & 13.3057 & 741455 & 3406.26\\
\hline
(2,45) & 7 & 5 & $(x+1)(x^2+x+1)$ & 6 & 32.9687 & $5.93164 \times 10^{6}$ & 8439.99\\
\hline
(2,51) & 7 & 4 & $x^{51}-1$ & 0 & 9.14684 & $4.74531 \times 10^7$ & $9.59166 \times 10^6$\\
\hline
(2,28) & 3 & 5 & $x+1$ & 2 & 66.6522 & 16384 & 4265.74\\
\hline
(2,44) & 3 & 6 & $x^{11}-1$ & 0 & 24.8377 & $4.1943 \times 10^6$ & 6358.45\\
\hline
(2,52) & 3 & 6 & $x^{13}-1$ & 0 & 22.0983 & $6.71089 \times 10^7$ &  5657.16\\
\hline
(2,56) & 15 & 6 & $x^{7}-1$ & 0 & 16.988 & $2.68435 \times 10^8$ & 17395.6\\
\hline
(2,60) & 15 & 9 & $x^{15}-1$ & 0 & 82.2883 & $1.07374 \times 10^9$ &  $5.39285 \time 10^6$\\
\hline
(2,22) & 3 & 3 & $x^{11}-1$ & 0 & 7.6329 & 2048 &  1954.02\\
\hline
(2,70) & 3 & 8 & $x^{35}-1$ & 0 & 24.8631 & $3.43597 \times 10^{10}$ &  $1.62943 \times 10^6$\\
\hline 
(4,18) & 15 & 6 & $(x+1)(x^2+x+1)$ & 1 & 42.1079 & 262144  & 42.1079 \\ 
\hline  
(4,14) & 3 & 5 & $x+1$ & 2 & 35.4555 &    16384 & 2269.15 \\ 
\hline
(4,20) & 3 & 6 & $x^5-1$ & 0 & 60.7588 &  $1.04858\times 10^6$    & 15554.3 \\ 
\hline 
(4,22) & 3 & 6 & $x^{11}-1$ & 0 & 24.8377 &  $4.1943 \times 10^6$ & 6358.45\\
\hline 
(4,28) & 3 & 7 & $x^7-1$ & 0 & 40.9888 &  $2.68435 \times 10^8$  & 41972.5 \\ 
\hline 
(4,30) & 15 & 9 & $x^{15}-1$ & 0 & 82.2883 & $1.07374 \times 10^9$ & $5.39285 \times 10^6$\\
\hline
(4,11) & 3 & 3 & $x^{11}-1$ & 0 & 7.6329 & 2048 & 1954.02\\
\hline
(4,13) & 3 & 2 & $x^{13}-1$ & 0 & 5.00293 & 8192 &  1280.75\\
\hline
(4,15) & 3 & 5 & $(x+1)(x^2+x+1)$ & 3 & 44.2638 & 32768 &  11332.7\\
\hline
(4,25) & 3 & 6 & $x^{25}-1$ & 0 & 16.8495 & $3.35544 \times 10^7$ & 17253.9\\
\hline
(4,35) & 33 & 7 & $x+1$ & 5 & 31.9641 & $3.43596 \times 10^{10}$ & 2045.7\\
\hline
(8,6) & 3 & 3 & $x+1$ & 1 & 14.7186 & 512 & 235.498\\
\hline
(8,12) & 15 & 6 & $x^3-1$ & 0 & 32.9531 & 262144 & 33744\\
\hline
(16,7) & 3 & 5 & $x+1$ & 2 & 30.8825 & 16384 & 1976.48\\
\hline
(128,3) & 7 & 2 & $x^3-1$ & 0 & 5.06649 & 1448.15 & 1297.02\\
\hline

\end{tabular}
\end{center}
\caption{Pairs $(q,m)$ appearing in the proof of Theorem~\ref{main2}, in which the corresponding test yielded a positive conclusion.\label{table2}}
\end{table}


 As an immediate consequence of the above results, we obtain Theorem~\ref{main2}.
%
%

\section{A few computational results} \label{sec:computations}

In this section we comment on the situation with the possible exceptional pairs that appear in Theorems~\ref{main1} and \ref{main2}. In particular, we wrote a script in \textsc{SageMath}, with the purpose of explicitly verifying whether the pairs in question are, in fact, genuine exceptions.

For every pair $(q,m)$, our script first fixes a primitive element $\alpha\in\Fm$ and then for every quintuple $a,b,c,d,e\in\Fm$ with $a\neq 0$ and $dx+e\neq 0$, it checks whether there exists some power $\alpha^i$ with $\gcd(i,q^m-1)=1$ of $\alpha$ (hence a primitive element), such that $\alpha^i$ is normal over $\F$ and $\frac{a\alpha^2 + b\alpha + c}{d\alpha +e}$ is primitive and normal over $\F$. For the primitivity check, we just compute the corresponding multiplicative order and for the normality check, we use \cite[Theorem~2.39]{10}. If this search is successful for every valid quintuple, then the pair $(q,m)$ is not an exception, while if it fails, even for one valid quintuple, the pair $(q,m)$ is a genuine exception.

Unfortunately, the extremely high number of such quintuples, even for ``small'' numbers, seems to create an impenetrable obstacle for a complete solution, with the exception of the pairs $(2,2)$ and $(2,3)$. On the other hand we managed to examine a respectable number of quintuples for all the other pairs and collected useful data.

In Table~\ref{table3} we present the pairs $(q,m)$, for which we found counter examples, while in Table~\ref{table4}, we present those for which we did not find counter examples.

\begin{table}[h]
  \begin{center}
    \begin{tabular}{|c|c|c|c|}
      \hline $(q,m)$ & $f$ & counter-example & checked/exceptional 5-ples \\ \hline
      $(2, 2)$ & $ x^2 + x + 1$ & $(\alpha, 0, 0, \alpha, \alpha)$ &  720/252 \\ \hline 
$(2, 3)$ & $ x^3 + x + 1$ & $(\alpha, 0, 0, \alpha, \alpha^2 + \alpha)$ &  28224/8295 \\ \hline 
$(2, 4)$ & $ x^4 + x + 1$ & $(\alpha, 0, 0, 0, \alpha^3 + \alpha^2)$ &  64513/22109 \\ \hline 
$(2, 5)$ & $ x^5 + x^2 + 1$ & $(\alpha, 0, \alpha, 0, \alpha)$ &  53345/52 \\ \hline 
$(2, 6)$ & $ x^6 + x^4 + x^3 + x + 1$ & $(\alpha, 0, 0, \alpha^2, \alpha^5 + \alpha^4 + \alpha + 1)$ &  21857/77 \\ \hline 
$(4, 2)$ & $ x^4 + x + 1$ & $(\alpha, 0, \alpha, \alpha, \alpha^3 + \alpha^2)$ &  266115/1985 \\ \hline 
$(4, 3)$ & $ x^6 + x^4 + x^3 + x + 1$ & $(\alpha, 0, 0, \alpha, \alpha)$ &  47708/152 \\ \hline 
    \end{tabular}
  \end{center}
  \begin{tiny}
     \textbf{Notes:}\begin{enumerate} \item $\alpha$ is a root of $f\in\F[x]$. \item For the pairs $(2,2)$ and $(2,3)$ the search was exhaustive. \end{enumerate}
  \end{tiny}
  \caption{Results of the computer test, where counter-examples were found.\label{table3}}
\end{table}

\begin{table}[h]
\begin{center}
\begin{tabular}{|c|c||c|c|}
\hline $(q,m)$ & checked 5-ples & $(q,m)$ & checked 5-ples \\ \hline
$(2, 7)$ & 11400  & 
$(2, 8)$ & 9067  \\ \hline
$(2, 10)$ & 12894  & 
$(2, 11)$ & 6504 \\ \hline
$(2, 12)$ & 1830  & 
$(2, 13)$ & 4765 \\ \hline
$(2, 14)$ &  2584  & 
$(2, 15)$ & 2003 \\ \hline
$(2, 16)$ &  2993  & 
$(2, 18)$ &  1104 \\ \hline
$(2, 20)$ & 1460  & 
$(2, 21)$ &  575 \\ \hline
$(2, 24)$ & 829  & 
$(2, 30)$ &  371 \\ \hline
$(4, 4)$ &  78278  & 
$(4, 5)$ &  25982 \\ \hline
$(4, 6)$ &  16072  & 
$(4, 7)$ &  19732 \\ \hline
$(4, 8)$ & 24391  & 
$(4, 9)$ & 4892  \\ \hline
$(4, 10)$ &  12001  & 
$(4, 12)$ &  4399 \\ \hline
$(8, 2)$ & 244944  & 
$(8, 3)$ &  163528 \\ \hline
$(8, 4)$ &  64654  & 
$(8, 5)$ &  67844  \\ \hline
$(8, 6)$ & 37279  & 
$(8, 7)$ &  11706  \\ \hline
$(8, 8)$ & 16416  & 
$(8, 14)$ &  1177  \\ \hline
$(16, 2)$ & 202189  & 
$(16, 3)$ & 79012  \\ \hline
$(16, 4)$ & 70934  & 
$(16, 5)$ & 28604  \\ \hline
$(16, 6)$ &  21965  & 
$(16, 15)$ &  235  \\ \hline
$(32, 2)$ &  189487  & 
$(32, 3)$ &  126253 \\ \hline
$(64, 2)$ & 133395  & 
$(64, 4)$ &  42129 \\ \hline
$(128, 2)$ & 163368  & 
$(256, 2)$ &  141196 \\ \hline
$(512, 2)$ &  135355  & 
$(1024, 2)$ &  106349  \\ \hline
\end{tabular}
\end{center}
\caption{Results of the computer test, where counter-examples were not found.\label{table4}}
\end{table}

Due to the large number of quintuples that we checked, without finding any counter-example, for the pairs $(q,m)$ that appear in Table~\ref{table4}, we believe that the only genuine exceptions to the problem we considered in this paper are the pairs that appear in Table~\ref{table3}. In other words, based on our computational evindence, we state the following.

\begin{Conjecture} \label{conj1}
Let $\mathbb{F}_{q^m}$ be a finite field of even characteristic. Then there exists an element $\alpha$ in $\mathbb{F}_{q^m}$, such that both $\alpha$ and  $f(\alpha)$ are simultaneously primitive normal in $\mathbb{F}_{q^m}$ over $\F$, where $f(x)=\frac{ax^2+bx+c}{dx+e}$, with $a,b,c,d,e\in \mathbb{F}_{q^m}$, $a\neq 0$, and $dx+e\neq 0$ unless $(q,m)$ is one of the pairs
$(2,2)$, $(2,3)$, $(2,4)$, $(2,5)$, $(2,6)$, $(4,2)$ or $(4,6)$, with these pairs being genuine exceptions.
\end{Conjecture}

\section{Some conjectures on rational functions} \label{sec6}

The following conjectures are extensions of the theorems given by Cohen, H.~Sharma and R.~Sharma in \cite{29}. For a finite field $\Fm$ and a rational function $f(x)\in \Fm(x)$, we denote by $\deg(f)$ the sum of the degrees of $f_1$ and $f_2$, if $f(x)=f_1/f_2$ and $f_1, f_2$ are relatively prime polynomials. The following conjectures are based on the results obtained during various experiments performed on similar rational forms as in this paper, some of which are studied briefly and will be discussed extensively in our next papers. Due to significantly large number of finite fields and very fragile behavior of its properties, a large scale analysis is required to establish our claims and this will be the focus of our subsequent study.

\begin{Conjecture} Take $f\in\Fm(x)$ and write $f=f_1/f_2$, where $f_1,f_2$ are relatively prime polynomials over $\Fm$. Let $n>2$ be the degree of $f$, such that $n=n_1+n_2$, where $n_1, n_2$ are degrees of $f_1$ and $f_2$ respectively. 
Then there exist an element $\alpha\in \Fm$ such that both $\alpha$ and $f(\alpha)$ are simultaneously primitive normal elements of $\Fm$ over $\F$ if 
$$ q^{m/2}> (2n-2)W(q^m-1)^2\Omega(x^m-1)^2, $$
provided the followings hold:
\begin{itemize}
\item[(i)] $f(x)$ is not of the form $ax^ig^h(x)$, where $i$ is an integer, $1\neq h\mid q^m-1$ and $a\in \Fm^*$.
\item[(ii)] If $n_1\neq n_2$, then $p\nmid n_2$, where $p$ is the characteristic of $\Fm$. 
\end{itemize} 
\end{Conjecture}

Further, one can apply the prime sieve, see Section~\ref{sec4}, to improve the above bound, leading to the next conjecture.
\begin{Conjecture} The sufficient condition for existence of an element $\alpha$ in $\Fm$ such that $(q,m)$ is a primitive normal pair is 
$q^{m/2}> (2n-2)W(d)^2\Omega(g)^2\mathfrak{S}$.
\end{Conjecture}

\end{document}